\documentclass[a4paper, 12pt]{amsart}
\usepackage{mathrsfs, amssymb, array}
\usepackage[all, matrix, arrow, curve]{xy}

\makeatletter
\@addtoreset{equation}{subsection}
\makeatother
\newcommand{\h}{\operatorname{H}}

\newcommand{\Spec}{\operatorname{Spec}}
\newcommand{\Diff}{\operatorname{Diff}}

\newcommand{\Sing}{\operatorname{Sing}}
\newcommand{\red}{\operatorname{red}}

\newcommand{\wt}{\operatorname{wt}}
\newcommand{\ord}{\operatorname{ord}}
\newcommand{\Sat}{\operatorname{Sat}}

\newcommand{\corank}{\operatorname{corank}}
\newcommand{\len}{\operatorname{len}}
\newcommand{\gr}{\operatorname{gr}}
\newcommand{\pgr}{\operatorname{pgr}}

\newcommand{\unit}{\operatorname{unit}}
\newcommand{\torsion}{\operatorname{torsion}}

\newcommand{\toplus}{\mathbin{\tilde\oplus}}
\newcommand{\totimes}{\mathbin{\tilde\otimes}}
\newcommand{\bars}{\overline{s}}

\newcommand{\sigmaord}{\sigma\mbox{-}\ord}

\newcommand{\CC}{\mathbb{C}}
\newcommand{\ZZ}{\mathbb{Z}}
\newcommand{\PP}{\mathbb{P}}
\newcommand{\QQ}{\mathbb{Q}}

\newcommand{\mm}{{\mathfrak{m}}}
\newcommand{\SSS}{{\mathscr{S}}}
\newcommand{\KKK}{{\mathscr{K}}}
\newcommand{\LLL}{{\mathscr{L}}}
\newcommand{\MMM}{{\mathscr{M}}}

\newcommand{\OOO}{\mathscr{O}}
\newcommand{\DDD}{\mathscr{D}}

\newcommand{\NNN}{\mathscr{N}}

\newcommand{\muu}{{\boldsymbol{\mu}}}

\newcommand{\type}[1]{$\mathrm{#1}$}

\renewcommand\labelenumi{{\rm (\roman{enumi})}}
\renewcommand\theenumi{(\roman{enumi})}

\newcounter{nn}
\renewcommand\thenn{{\rm(\arabic{section}.\arabic{subsection}.\arabic{nn})}}
\newcommand{\nnn}{\refstepcounter{nn}\thenn}

\newtheorem{thm}[subsection]{}

\newtheorem{sthm}[equation]{}

\newtheorem*{lemma*}{Lemma}
\newtheorem*{claim*}{Claim}
\newtheorem*{corollary*}{Corollary}

\theoremstyle{definition}
\newtheorem{de}[subsection]{}
\newtheorem{sde}[equation]{}

\newtheorem*{remark*}{Remark}

\newenvironment{lemma}{\begin{thm}{\bf Lemma.}}{\end{thm}}

\newenvironment{slemma}{\begin{sthm}{\bf Lemma.}}{\end{sthm}}

\newenvironment{remark}{\begin{thm}{\bf Remark.}}{\end{thm}}
\newenvironment{sremark}{\begin{sde}{\bf Remark.}}{\end{sde}}
\newenvironment{sclaim}{\begin{sthm}{\bf Claim.}}{\end{sthm}}
\newenvironment{computation}{\begin{thm}{\bf Computation.}}{\end{thm}}

\newenvironment{sconstruction}{\begin{sde}{\bf Construction.}}{\end{sde}}
\newenvironment{assumption}{\begin{de}{\bf Assumptions.}}{\end{de}}

\newcommand{\flabel}[2]{\text{(\ref{#1}#2)}}
\newcommand{\fref}[2]{{\rm (\ref{#1}#2)}}
\newcommand{\xref}[1]{{\rm\ref{#1}}}
\input{cyracc.def}

\title{Threefold extremal contractions \\ of type \type{(IIA)}, I}
\author{Shigefumi Mori}

\address{S.~Mori: RIMS,
Kyoto University, Oiwake-cho, Kitashirakawa, Sakyo-ku, Kyoto
606-8502, Japan}
\email{mori@kurims.kyoto-u.ac.jp}

\author{Yuri Prokhorov}
\thanks{The first author's work partially supported by JSPS KAKENHI Grant
Numbers (A) 25287005 and (S) 24224001.
\newline\indent
The second author's work partially supported by the RFFI grants
15-01-02164a,
15-01-02158a,
15-51-50045{\font\tencyr=wncyr10\tencyr\cyracc YAF\_a},
and
by a subsidy granted to the HSE by the Government 
of the Russian Federation for the implementation of 
the Global Competitiveness Program.
}
\address{Y.~Prokhorov:
Steklov Mathematical Institute,
8 Gubkina street, Moscow 119991, Russia, and
\newline\indent
Department of Algebra, 
Moscow State University, and
\newline\indent
National Research University Higher School of Economics
}
\email{prokhoro@mi.ras.ru}

\begin{document}

\begin{abstract}
Let $(X, C)$ be a germ of a threefold $X$ with terminal singularities
along an irreducible reduced complete curve $C$
with a contraction $f: (X, C)\to (Z, o)$
such that $C=f^{-1}(o)_{\red}$ and $-K_X$
is ample. Assume that $(X, C)$ contains a point
of type \type{(IIA)} and that a general member
$H\in |\OOO_X|$ containing $C$ is normal.
We classify such germs
in terms of $H$.
\end{abstract}
\maketitle

\section{Introduction}
Recall that a \emph{contraction} is a proper surjective morphism $f: X\to Z$ of normal varieties
such that $f_*\OOO_X=\OOO_Z$.
Let $(X,C)$ be the analytic germ of a threefold with terminal singularities
along a reduced complete curve. We say that $(X,C)$
is an \textit{extremal curve germ} if
there is a contraction $f: (X,C)\to (Z,o)$ such that
$C=f^{-1}(o)_{\red}$ and $-K_X$
is $f$-ample.
Furthermore, $f$ is called \textit{flipping} if
its exceptional locus coincides with $C$ and
\textit{divisorial}
if its exceptional locus is two-dimensional.
If $f$ is not birational, then $Z$ is a surface and
$(X,C)$ is said to be
a \textit{$\QQ$-conic bundle germ} \cite{Mori-Prokhorov-2008}.
Extremal curve germs appear naturally in the three-dimensional minimal model program
\cite{Mori-1988}.

In this paper we consider only extremal curve germs with \textit{irreducible}
central fiber $C$.
All the possibilities for the local behavior of $C$
near singular points of $X$ are classified into types
\type{(IA)}, \type{(IC)}, \type{(IIA)}, \type{(IIB)},
\type{(IA^\vee)}, \type{(II^\vee)}, \type{(ID^\vee)}, \type{(IE^\vee)},
and \type{(III)}, for
whose definitions we
refer the reader to \cite{Mori-1988} and \cite{Mori-Prokhorov-2008}.

In this paper we study extremal curve germs
containing points of type \type{(IIA)}.
As in \cite{Kollar-Mori-1992}, \cite{Mori-Prokhorov-2010}, and
\cite{Mori-Prokhorov-2011}
the classification is done in terms of
a general divisor $H$ of $|\OOO_X|_C$,
the linear subsystem of $|\OOO_X|$ consisting of members containing $C$.
Flipping extremal curve germs of type \type{(IIA)}
were classified in \cite[ch. 7]{Kollar-Mori-1992}.
Our main result is the following.
\begin{thm}{\bf Theorem.}\label{main}
Let $(X,C)$ be an extremal curve germ
and let $f: (X, C)\to (Z,o)$ be the corresponding
contraction.
Assume that $(X,C)$ is not flipping and it has a point $P$ of type \type{(IIA)}.
Furthermore, assume that the general member $H\in |\OOO_X|_C$ is normal.
Then $H$ has only rational singularities. Moreover, the following are the only possibilities
for the dual graph of $(H,C)$, and all the possibilities do occur.

\begin{enumerate}
\renewcommand\labelenumi{{\rm (\arabic{section}.\arabic{subsection}.\arabic{enumi})}\refstepcounter{equation}}
\renewcommand\theenumi{(\arabic{section}.\arabic{subsection}.\arabic{enumi})}

\item
\label{main-theorem-1.1.1}
$f$ is divisorial, $f(H)\ni o$ is of type \type{A_1},
\begin{equation*}
\xy
\xymatrix@R=3pt@C=17pt{
& &&&\circ\ar@{-}[d]
\\
\circ\ar@{-}[r]&\circ\ar@{-}[r]&\bullet\ar@{-}[r]&\underset4\circ\ar@{-}[r]
&\circ\ar@{-}[r]&\underset4\circ
}
\endxy
\end{equation*}

\item
\label{main-theorem-1.1.2}
$f$ is divisorial, $f(H)\ni o$ is of type \type{A_1}, 
\begin{equation*}
\xy
\xymatrix@R=3pt@C=17pt{
&&&\overset3{\circ}\ar@{-}[d]
\\
\circ\ar@{-}[r]&\bullet\ar@{-}[r]&\underset3\circ\ar@{-}[r]
&\circ\ar@{-}[r]&{\circ}\ar@{-}[r]&\underset4\circ
}
\endxy
\end{equation*}

\item
\label{main-theorem-1.1.3}
$f$ is divisorial, $f(H)\ni o$ is of type \type{D_{5}},
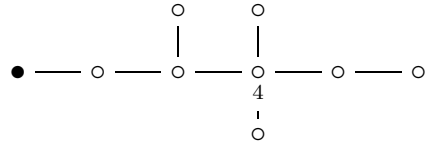
\begin{equation*}
\xy
\xymatrix@R=7pt@C=17pt{
&&\circ\ar@{-}[d]&\circ\ar@{-}[d]
\\
\bullet\ar@{-}[r]&\circ\ar@{-}[r]&\circ\ar@{-}[r]
&\underset4\circ\ar@{-}[r]&\circ\ar@{-}[r]&\circ
\\
&&&\circ\ar@{-}[u]
}
\endxy
\end{equation*}

\item
\label{main-theorem-conic-bundle-1.1.4}
$f$ is a $\QQ$-conic bundle, 
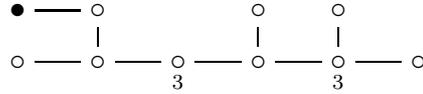
\begin{equation*}
\xy
\xymatrix@R=3pt@C=17pt{
\bullet\ar@{-}[r]&\circ\ar@{-}[d]&&\circ\ar@{-}[d]&\circ\ar@{-}[d]
\\
\circ\ar@{-}[r]&\circ\ar@{-}[r]
&\underset3\circ\ar@{-}[r]&\circ\ar@{-}[r]&\underset3\circ\ar@{-}[r]&\circ
}
\endxy
\end{equation*}
\end{enumerate}
In cases \xref{main-theorem-1.1.3} and \xref{main-theorem-conic-bundle-1.1.4}
$P$ is the only singular point of $X$.
In other cases $X$ can have at most one extra type \type{(III)} point.
\end{thm}

The case of non-normal $H$ will be treated in a forthcoming paper.
The main theorem is a consequence of the
technical theorem \ref{main-1} which will be proved in \S\S\ref{section-lP=1-Sing=P}-\ref{section-lP=5}.
The proof splits into cases
according to the invariant $\ell(P)$ (see \ref{equation-iP}):
\begin{center}
\begin{tabular}{c@\quad |@\quad c@\quad c@\quad c@\quad c}
$\ell(P)$&$1$&$3$&$4$&$5$
\\
\hline
\\[-9pt]
Sect. 
&
\S \ref{section-lP=1-Sing=P}-\ref{section-lP=1-Sing-ne=P}&
\S \ref{section-lP=3}
&
\S \ref{section-lP=4}
&
\S \ref{section-lP=5}
\end{tabular}
\end{center}

\section{Notation and the technical theorem}\label{section-Preliminaries}
\begin{de}{\bf Set-up.}
\label{Set-up}
Let $(X,C)$ be an extremal
curve germ and let $f: (X, C)\to (Z,o)$ be the corresponding
contraction.
The ideal sheaf of $C$ in $X$ we denote by $I_C$ or simply by $I$.
Assume that $(X,C)$ has a point $P$ of type \type{(IIA)}.
Then by \cite[6.7, 9.4]{Mori-1988} and
\cite[8.6, 9.1, 10.7]{Mori-Prokhorov-2008} $P$ is the only non-Gorenstein point of $X$
and $(X,C)$ has at most one Gorenstein singular point $R$ \cite[6.2]{Mori-1988},
\cite[9.3]{Mori-Prokhorov-2008}.
Furthermore, assume that the germ $(X,C)$ is not flipping.
\begin{sde}\label{sde}
We have two possibilities:
\begin{itemize}
\item
$f$ is a $\QQ$-conic bundle and $(Z,o)$ is smooth \cite[Th. 1.2]{Mori-Prokhorov-2008};
\item
$f$ is a divisorial contraction and $(Z,o)$ is a cDV point
(or smooth) \cite[Th. 3.1]{Mori-Prokhorov-2010}. Moreover, in this case
$(Z,o)$ is either smooth, \type{cA} or \type{cD} point (because
$|-K_{Z}|$ has a Du Val member of type \type{D}, see \ref{ge} below).
\end{itemize}
\end{sde}
\begin{sde}\label{sde1}
Let $H$ be a general member of $|\OOO_X|_C$ through $C$ and let $H_Z=f(H)$. 
Let $H^{\mathrm n}\to H$ be the normalization (we put $H^{\mathrm n}= H$ if $H$ is normal). Then the composition map 
$H^{\mathrm n}\to H_Z$ has connected fibers. Moreover,
it is a rational curve fibration if $\dim Z=2$
and it is a birational contraction to a point $(H_Z, o)$ 
which is either smooth or Du Val point of type \type{A} or \type{D}
(see \ref{sde}) if $f$ is divisorial. 
In all cases $H^{\mathrm n}$ has only rational singularities.
\end{sde}
\begin{sde}
Throughout this paper $(X^\sharp, P^\sharp)\to (X,P)$
denotes the index-one cover. For any object $V$ on $X$
we denote by $V^\sharp$ the pull-back of $V$ on $X^\sharp$.
\end{sde}
\end{de}

\begin{de}\label{equation-iP}
Denote
\[
\ell(P):=\len_P I^{\sharp (2)}/I^{\sharp 2},
\]
where $I^\sharp$ is the ideal defining $C^\sharp$ in $X^\sharp$
and $\KKK^{(n)}$ is the symbolic $n$-th power of a prime 
ideal $\KKK$.
Recall (see \cite[(2.16)]{Mori-1988}) that in our case
\begin{equation}\label {equation-iP-lP}
i_P(1)=\lfloor(\ell(P)+6)/4\rfloor.
\end{equation}
In our proof of the main result we distinguish cases according
to the value of $\ell(P)$ and treat these cases separately
in the next sections.
\end{de}

\begin{de}
\label{(7.5)}
According to \cite[A.3]{Mori-1988} we can express the \type{(IIA)} point as
\begin{equation}
\label{equation-XC}
\begin{array}{rcl}
(X, P)&=&
\{\alpha=0\}/\muu_4(1, 1, 3, 2)\subset\CC^4_{y_1,\dots, y_4}/\muu_4(1, 1, 3, 2),
\\[4pt]
C&=&\{y_1\text{-axis}\}/\muu_4,
\end{array}
\end{equation}
where $\alpha=\alpha(y_1,\dots, y_4)$ is a semi-invariant such that
\begin{equation}\label{equation-alpha}
\wt\alpha\equiv 2\mod 4,\qquad \alpha\equiv y_1^{\ell(P)}y_j\mod (y_2, y_3, y_4)^2,
\end{equation}
where $j= 2$ (resp. $3$, $4$) if $\ell(P)\equiv 1$ (resp. $3$, $0$) $\mod 4$
\cite[(2.16)]{Mori-1988} and $(I^\sharp)^{(2)}=(y_j)+(I^\sharp)^{2}$.
Moreover, $y_2^2,\, y_3^2\in \alpha$ (because $(X,P)$ is a terminal point of type
\type{cAx/4}).
Note that $\ell(P)\not\equiv 2\mod 4$ because
of the lack of a variable with $\wt\equiv 0\mod 4$.
\end{de}

\begin{de}\label{notation-grn}
Let $(X,C\simeq\PP^1)$ be an extremal curve germ
and let $f: (X, C)\to (Z,o)$ be the corresponding
contraction.
For every $n\ge 1$, we define an $\OOO_C$-module
\[
\gr_C^n\OOO := I^{(n)}/I^{(n+1)}. 
\]
Recall that $R^1f_*\OOO_X=0$ (see \cite[1.2]{Mori-1988}). 
Hence $\h^1(\gr_C^1\OOO)=0$ 
\end{de}

Now we are ready to formulate our main result.

\begin{thm}{\bf Theorem.}\label{main-1}
In the notation of \xref {notation-grn} assume that 
$\h^0(\gr^1_C\OOO)\neq 0$.
Furthermore, assume that $(X,C)$ is not flipping and it has a point $P$ of type \type{(IIA)}.
Then the general member $H\in |\OOO_X|_C$ is normal and
has only rational singularities. Moreover, the following are the only possibilities
for the dual graph $\Delta(H,C)$ of $(H,C)$, and all the possibilities do occur.
\smallskip
{\rm \begin{center}
\begin{tabular}{l|c|c|c|c}

No. & $f$ & $f(H)\ni o$ & $\Delta(H,C)$ & Ref.
\\[7pt]
\hline
&&&
\\[-3pt]
\nnn\label{1-main-theorem-1.1.1} & divisorial&\type{A_1} &\xref{main-theorem-1.1.1}&
\fref{equation-division-cases-(7.8.2.2)}{b}, \fref{cases-gr1COOO-S4}{a}
\\
\nnn\label{1-main-theorem-1.1.2}&
divisorial&\type{A_1}&
\xref{main-theorem-1.1.2}&

\fref{equation-possibilities-(7.8.1.1)-a}{a},
\fref{equation-possibilities-(7.8.1.1)}{b}

\\
\nnn\label{1-main-theorem-1.1.3}&
divisorial&\type{D_{5}}&
\xref{main-theorem-1.1.3}&

\fref{equation-lP=3-gr2O-1}{a, c}, \fref{equation-possibilities-lP=5-gr1COOO}{a}

\\
\nnn\label{1-main-theorem-conic-bundle-1.1.4}&
$\QQ$-conic bundle&&
\xref{main-theorem-conic-bundle-1.1.4}& 
\fref{equation-lP=4-grOJ}{b},
\fref{equation-possibilities-lP=5-gr1COOO}{b}
\end{tabular}
\end{center}}
\smallskip
\noindent
In cases \xref{1-main-theorem-1.1.3} and \xref{1-main-theorem-conic-bundle-1.1.4}
$P$ is the only singular point of $X$.
In other cases $X$ can have at most one extra type \type{(III)} point.
\end{thm}

\begin{thm}{\bf Corollary}
In the notation of \xref{main-1}
a general member $H\in |\OOO_X|_C$ is normal
if and only if $\h^0(\gr^1_C\OOO)\neq 0$.
\end{thm}

\begin{remark}
A general member $H$ may be non-normal even if $\h^0(\gr^1_C\OOO)\neq 0$, e.g. in case of \type{(IA)} \cite{Mori-Prokhorov-2010}.
\end{remark}

\section{Preliminary results}
\begin{de}\label{ge}
Recall that in our case
a general member $D\in |-K_X|$ does not contain $C$
\cite[Th. 7.3]{Mori-1988},
\cite[Prop. 1.3.7]{Mori-Prokhorov-2008}.
Hence $D\cap C=\{P\}$ and $D$ has at $P$ a singularity of type \type{D_{2n+1}}
\cite[6.4B]{Reid-YPG1987}.
Note that $D\simeq f(D)$ if $f$ is birational and
$f_D: D\to Z$ is finite of degree $2$ if $f$ is a $\QQ$-conic bundle.
In the coordinates $y_1,\dots,y_4$, the divisor $D$ is given by
\[
D=\{y_1= \xi\}/\muu_4,\qquad \xi\in (y_2,\, y_3,\, y_4).
\]
\end{de}

\begin{slemma}\label{lemma-extend-sections-D}
We have a natural surjection
$\h^0(\OOO_X) \twoheadrightarrow \OOO_D \mod (y_2,y_3,y_4)^3$.
In particular, $y_4^2$ and $y_2y_3$ appear with arbitrary coefficients
in a member $\beta \in \h^0(I)$.
\end{slemma}

\begin{proof}
Consider the following exact sequence
\[
0\longrightarrow \omega_X \longrightarrow \OOO_X
\longrightarrow \OOO_D \longrightarrow 0
\]
(cf. \cite[Th. 1.2]{Mori-1988}).
If $f$ is a birational contraction, then
$R^1f_*\omega_X=0$ by the Grauert-Riemenschneider vanishing theorem.
Hence any section $\bar s\in \OOO_D$
lifts to a section $s\in f_*\OOO_X$.
So, the assertion is clear in this case.
Assume that $f$ is a $\QQ$-conic bundle.
Let $\tau=f_D : D\to Z$ be the induced double cover.
Obviously, $\tau:=f|_D$ is 
of degree $2$.
Since $R^1f_*\omega_X=\omega_Z$ (see \cite[Lemma 4.1]{Mori-Prokhorov-2008})
and $\omega_D\simeq\OOO_D$, we have
the following split exact sequence
\[
0 \longrightarrow
f_*\OOO_X \longrightarrow \tau_* \omega_D \xrightarrow{\operatorname{Tr}_{D/Z}}
\omega_Z\longrightarrow 0,
\]
where by smoothness of $Z$ the splitting homomorphism 
$\omega_Z\to \tau_*\omega_D$ is induced by 
\[
\tau^*\Omega_Z^2=\tau^*\omega_Z\subset \Omega_D^2\longrightarrow
\omega_D.
\]
Hence, as in \cite[2.1-2.2]{Mori-Prokhorov-2008III},
we have a surjection
\begin{equation}\label{20120304}
f_*\OOO_X \longrightarrow \tau_* (\omega_D/\tau^*\omega_Z).
\end{equation}
We may write the equation of $D^\sharp$
in $\CC^3_{y_2,y_3,y_4}$ as follows:
\[
\gamma(y_2,y_3,y_4):=\alpha(\xi,y_2,y_3,y_4)
\in (y_2,y_3,y_4)^2.
\]
Locally, near $P^\sharp$, the sheaf $\omega_{D^\sharp}$ is generated by
\[
\eta:=
\operatorname{Res}\frac{dy_2\wedge d y_3\wedge d y_4}{\gamma}
=
-\frac{dy_2\wedge d y_3}{\partial \gamma/\partial y_4}
=
\frac{dy_2\wedge d y_4}{\partial \gamma/\partial y_3}
=
-\frac{dy_3\wedge d y_4}{\partial \gamma/\partial y_2}.
\]
Since $\eta$ is an invariant, it is also a generator of $\omega_{D}$
near $P$.
Since the coordinates $u_i$ of $(D,P)$ belong to
$(y_2,y_3,y_4)^2$
and
$\Omega_D^2$ is generated by the elements
$d u_i\wedge d u_j$, we see that
$\tau^*\Omega_Z^2=\eta \Im$, where
\[
\Im^\sharp \subset
(y_2,y_3,y_4)^2 \left(\partial \gamma/\partial y_4,\
\partial \gamma/\partial y_3,\
\partial \gamma/\partial y_2\right) \subset (y_2,y_3,y_4)^3.
\]
Finally we note
that $y_4^2$ and $y_2y_3$ are independent
modulo
$\bigl((y_2,y_3,y_4)^3\OOO_{D^\sharp}\bigr)^{\muu_4}$.
\end{proof}

\begin{slemma}\label{lemma-extend-sections-E}
Suppose we are given a general member $E \in |-K_X|$
containing $C$ 
and that it is defined by $y_2=\xi$ near $P$, where $\xi \in (y_3,y_4)
\cap (y_1,y_3,y_4)^2$.
Then we have a surjection
\[
\h^0(\OOO_X) \twoheadrightarrow
\operatorname{Im} \left(\h^0( \OOO_E)\to \OOO_{E,P}/\NNN\right),
\]
where $\NNN^\sharp:=\bigl(y_3(y_1,y_3,y_4)^2+(y_1,y_3,y_4)^4\bigr) \subset \OOO_{E^\sharp, P^\sharp}$
defines $\NNN\subset \OOO_{E,P}$.
\end{slemma}
\begin{proof}
If $f$ is birational we have a surjection $\h^0(\OOO_X) \twoheadrightarrow
\h^0( \OOO_E)$,
otherwise we apply the same argument as in the proof of the previous lemma.
Note that the map \eqref{20120304} is still surjective because
the proof of \cite[2.1-2.2]{Mori-Prokhorov-2008III} works even
if $\tau:=f_E: E\to Z$ is not finite but generically finite.

We may assume that the equation of $E^\sharp$
in $\CC^3_{y_1,y_3,y_4}$ is as follows:
\[
\gamma(y_1,y_3,y_4):=\alpha(y_1,\xi,y_3,y_4)
\equiv c y_3^2 \mod (y_1,y_3,y_4)^3,
\]
where $c \in \CC^*$.
Note that
$\partial \gamma/\partial y_i \in (y_3)+(y_1,y_3,y_4)^2$.
As in the previous proof, we 
have a generator $\eta$ of $\omega_E^\sharp$ and
$\tau^*\Omega_Z^2=\Im \omega_E$, where
\[
\Im^\sharp 
\subset
(y_1,y_3,y_4)^2 \left(\partial \gamma/\partial y_4,\
\partial \gamma/\partial y_3,\
\partial \gamma/\partial y_1\right) 
\subset \NNN^\sharp.
\]
We are done because of the surjection
$\h^0(\OOO_X) \twoheadrightarrow \h^0(\OOO_E/\Im)$.
\end{proof}

\begin{de}\label{definition-J}
The techniques of \cite{Mori-1988} will be used freely.
For convenience of references, we recall a few facts from \cite[ch. 8]{Mori-1988}.
An ideal $J\subset \OOO_X$ is said to be \textit{laminal} if
it is an $I$-primary ideal and $I^{(2)}\not\supset J$. The \textit{width} of $J$ is the smallest
$d$ such that $J\supset I^d$. In this situation, define
\[
F^n(\OOO,J):=\Sat(J^qI^r+J^{q+1}),
\quad
\gr^n(\OOO,J):= F^n(\OOO,J)/F^{n+1}(\OOO,J),
\]
where $q:=\lfloor n/d \rfloor$, $r:=n-dq$. In particular, if $d=2$, then
$F^1(\OOO,J)=I$, $F^2(\OOO,J)=J$,
$F^3(\OOO,J)=\Sat(JI)$, $F^4(\OOO,J)=J^{(2)}$.
Further, under above assumptions there exists a natural \textit{saturated} filtration
\[
\gr^n(\OOO,J)=\Phi^0\gr^n(\OOO,J)\supset \cdots \supset\Phi^{q}\gr^n(\OOO,J)\supset
\Phi^{q+1}\gr^n(\OOO,J)=0
\]
such that each quotient
\[
\gr^{n,i}(\OOO,J):= \Phi^i\gr^n(\OOO,J)/\Phi^{i+1}\gr^n(\OOO,J)
\]
is a torsion free $\OOO$-module of rank $1$ \cite[8.6]{Mori-1988}. In particular,
we have $\gr^{2,0}(\OOO,J)=J/J\cap I^{(2)}$ and 
an $\ell$-exact sequence \cite[8.2.2)]{Mori-1988}
\begin{equation}
\label{equation-exact-sequence}
0\longrightarrow \gr^{2,1}(\OOO,J) \longrightarrow\gr^2(\OOO,J)
\xrightarrow{\alpha_J} \gr^{2,0}(\OOO,J) \longrightarrow 0.
\end{equation}
\end{de}

\begin{lemma}\label{construction-s}
Let $(X,C)$ be an extremal curve germ and let $D\in |-K_X|$
be a general member as in \xref{ge}. 
Let $J$ be as in \xref{definition-J} with $d=2$.
Assume $\gr^1 (\OOO,J)=(-1+bP^\sharp)$ with $0\le b \le3$. Then the natural map
\[
\h^0\bigl( F^2(\OOO(D),J)\bigr) \longrightarrow \gr^2\bigl(\OOO(D),J\bigr)\totimes \CC_{P}
\]
is surjective.
\end{lemma}
\begin{proof}
By our hypothesis $\h^1(F^2(\OOO,J))=0$. So there is a surjection
\[
\h^0\bigl( F^2(\OOO(D),J)\bigr) \longrightarrow \h^0\bigl( F^2(\OOO(D),J)/ F^2(\OOO,J)\bigr)
= J(D)/J.
\]
Since $J^\sharp\subset J^\sharp(D^\sharp)\mm_{P^\sharp}$,
we have
\begin{multline*}
J^\sharp(D^\sharp)/J^\sharp \twoheadrightarrow
J^\sharp(D^\sharp)/ J^\sharp(D^\sharp)\mm_{P^\sharp}+
\Sat(I^\sharp J^\sharp)(D^\sharp)=
\\
=\gr^2\bigl(\OOO(D),J\bigr)^\sharp\totimes_{\OOO_{P^\sharp}} \CC_{P^\sharp}.
\end{multline*}
Taking the $\muu_4$-invariant part, we obtain our statement.
\end{proof}

\begin{thm}{\bf Lemma.}\label{case-up-1-s-in--KX}
Under the assumptions of \xref{construction-s}, suppose
$\gr^2 (\OOO,J)=(aP^\sharp)\toplus (-1+3P^\sharp)$ with $a\ge 0$. 
Then there is a global section $s$ of $F^2(\OOO(D),J)$ such that
$E=\{s=0\} \in |-K_X|_C$ induces an $\ell$-isomorphism
\begin{equation}\label{equation-sect-2-gr2OJ}
\gr^2(\OOO,J)= (aP^\sharp)\toplus\OOO_C(-E).
\end{equation}
\end{thm}
\begin{proof}
We have by Lemma \ref{construction-s} a global section $s$ of
$F^2(\OOO(D),J)$ inducing a nowhere vanishing
section $\bars=(\unit) \cdot (y_2+\cdots )/y_1$ of $\gr^2(\OOO(D),J)$ at $P$
so that there is an $\ell$-isomorphism
\[
\gr^2\bigl(\OOO_X(D),J\bigr)= \bigl((a+1)P^\sharp\bigr)\toplus\OOO_C\cdot \bars,
\]
which is \eqref{equation-sect-2-gr2OJ}$\totimes (P^\sharp)$.
\end{proof}

\begin{sde}\label{define-pF}
Assumptions as in Lemma \xref{case-up-1-s-in--KX}.
For $n\in\ZZ$, let $q:=\lfloor n/2\rfloor$ and $r:=n-2q\in \{0,\, 1\}$.
Let
\[
\begin{array}{l}
pF^n(\OOO_E,J):=
\OOO_EJ^qI^r \quad\text{outside of $P$,}
\\[5pt]
pF^n(\OOO_E,J)^\sharp:=
\OOO_{E^\sharp}J^{\sharp q}I^{\sharp r},\quad
pF^n(\OOO_E,J)=(pF^n(\OOO_E,J)^\sharp)^{\muu_4}\quad
\text{at $P$.}
\end{array}
\]
This defines $pF^n(\OOO_E,J)$ with $\ell$-structure.
Thus for $m>n$ the sheaf
\[
\pgr^{n,m}(\OOO_E,J):=pF^n(\OOO_E,J)/pF^{m}(\OOO_E,J)
\]
has an induced $\ell$-structure. We omit $m$ for simplicity if $m=n+1$. 
\end{sde}

\begin{thm}{\bf Lemma.}\label{grnE-formulae}
Assumptions as in Lemma \xref{case-up-1-s-in--KX}.
We have
\begin{eqnarray}
\label{F12E}
pF^1(\OOO_E,J)&=&I\OOO_E, \hspace{13pt} pF^2(\OOO_E,J)=J\OOO_E,
\\
\label{gr1E}
\pgr^1(\OOO_E,J)&=&\gr^1(\OOO,J)= (-1+bP^\sharp),
\\
\label{gr2E}
\qquad\pgr^2(\OOO_E,J)/(\torsion)&=&\gr^2(\OOO,J)/\OOO_C(-E)=(aP^\sharp),
\end{eqnarray}
and an $\OOO_C$-homomorphism which is generically an isomorphism:
\begin{equation}
\label{grnE}
\pgr^2(\OOO_E,J)^{\totimes q} \totimes \pgr^1(\OOO_E,J)^{\totimes r}
\to
\pgr^n(\OOO_E,J)/(\torsion).
\end{equation}
If $b=2$ \textup(resp. $b=1$\textup), 
then for some $\ell$-coordinates \textup(see \xref{(7.5)}\textup) 
we have $I^\sharp \OOO_{E^\sharp}=(y_3,y_4)$,
$J^\sharp \OOO_{E^\sharp}=(y_3,y_4^2)$
\textup(resp. $J^\sharp \OOO_{E^\sharp}=(y_4)$ and $a\equiv 2\mod 4$\textup) and
\begin{equation}\label{equation-new-pgr2}
\pgr^2(\OOO_{E},J) \otimes \CC_P =
\begin{cases}
\CC\cdot y_1y_3 \oplus \CC\cdot y_4^2 \simeq \CC^2& \text{if $b=2$,}
\\ 
\CC\cdot y_1^2 y_4 \simeq \CC &\text{if $b=1$}.
\end{cases}
\end{equation}
\end{thm}

\begin{proof}
By $\OOO_X(-E) \subset J$, we see \eqref{F12E},
\eqref{gr1E} and
$J\OOO_E/IJ\OOO_E\simeq J/(IJ+\OOO_X(-E))$.
Since $J/(\Sat(IJ)+\OOO_X(-E))=\gr^2(\OOO,J)/\OOO_C(-E)$,
we have \eqref{gr2E}.
By \eqref{F12E},
the homomorphism \eqref{grnE} is induced.
At a general point $Q \in C$, we can find $u,v \in I$ such that $I=(u,v)$, $J=(u,v^2)$,
and $E=\{uw=v^2\}$ or $\{ u=wv^2\}$ for some $w \in \OOO_{X,Q}$.
We have
$\pgr^{n}(\OOO_E,J)=\OOO_C\cdot u^qv^r$
at $Q$ in the former case and
$\pgr^{n}(\OOO_E,J)=\OOO_C\cdot v^n$ in the latter. Whence \eqref{grnE} follows.
When $b=1$, we have $j\neq 3$ in 
\eqref{equation-alpha} 
and 
$\alpha=y_3^2\cdot (\unit) +y_4\cdot (\cdots)$. Whence $J^\sharp \OOO_E^\sharp= (y_3^2, y_4)=(y_4)$,
and so $a\equiv 2\mod 4$ by \eqref{gr2E}.
\end{proof}

\begin{sthm}{\bf Corollary.}\label{surjection-Fn-to-grnE}
For every $m> n \ge 0$, we have
$\h^1(\pgr^n(\OOO_E,J))=0$ and
a natural surjection
\[
\h^0(pF^n(\OOO_E,J)) \twoheadrightarrow \h^0(\pgr^{n,m}(\OOO_E,J)).
\]
\end{sthm}
\begin{proof}
By \eqref{grnE} the sheaf $\pgr^n(\OOO_E,J)/(\torsion)$
is invertible of degree $\ge -1$
and we have $\h^1\bigl(\pgr^n(\OOO_E,J)\bigr)=0$ by the exact sequence
\[
0 \to (\torsion) \to \pgr^n(\OOO_E,J) \to
\pgr^n(\OOO_E,J)/(\torsion) \to 0.
\]
Hence, for every $m >n$,
the exact sequence
\[
0 \longrightarrow
\pgr^m(\OOO_E,J) \longrightarrow
\pgr^{n,m+1} (\OOO_E,J)
\xrightarrow{\nu_{n,m}}
\pgr^{n,m} (\OOO_E,J)
\longrightarrow
0
\]
induces a surjection
\[
\h^0(\nu_{n,m}) :
\h^0\left(\pgr^{n,m+1} (\OOO_E,J)\right)
\twoheadrightarrow
\h^0\left(\pgr^{n,m} (\OOO_E,J)\right).
\]
If we denote the completion of $E$ along $C$ by $E^{\wedge}$,
then we have a surjection
$\h^0\bigl(E^{\wedge},pF^n(\OOO_E,J)\OOO_{E^{\wedge}}\bigr) \twoheadrightarrow
\h^0\bigl(\pgr^{n,m}(\OOO_E,J)\bigr)$.
By the formal function theorem,
we have approximating global sections of
$\h^0(pF^n(\OOO_E,J))$ and the required surjection.
\end{proof}

\begin{sthm}{\bf Corollary.}\label{y12y4-in-beta}
For general $\lambda,\, \mu\in \CC$, we have
$\beta\in \h^0(J)$ such that 
\[
\beta=\begin{cases}
\cdots + \lambda y_1y_3+\mu y_4^2+\cdots&\text{if $b=2$,}
\\
\cdots + \lambda y_1^2y_4+\mu y_4^2+\cdots&\text{if $b=1$.}
\end{cases}
\]
\end{sthm}

\begin{proof}
When $b=2$, we have a surjection 
$\h^0(J\OOO_{E}) \twoheadrightarrow \pgr^2 (\OOO_E, J)$
by Corollary \ref{surjection-Fn-to-grnE}, and 
there exists an element $\bar \beta\in\h^0(J\OOO_E)$ sent to 
$\lambda y_1y_3+\mu y_4^2\in \pgr^2(\OOO_E,J)\otimes \CC_P$
by \eqref{gr2E} and \eqref{equation-new-pgr2}.
By Lemma \ref{lemma-extend-sections-E} there exists $\beta\in \h^0(\OOO_X)$ 
sent to the image of 
$\bar \beta$ in $\OOO_E/\NNN$.
Since $\bar \beta (P)=0$, we have $\beta\in \h^0(I)=\h^0(J)$.
The case is settled because 
the substitution $y_2=\xi(y_1,y_3,y_4)$ 
in \ref{lemma-extend-sections-E} does not affect the coefficients of 
$y_1y_3$, $y_4^2$, which are a part of a standard basis 
of $(\OOO_E/\NNN)^\sharp$.

Assume that $b=1$.
We have 
$\pgr^2(\OOO_E,J)=(aP^\sharp)$ and 
$\pgr^4(\OOO_E,J) \hookleftarrow (2aP^\sharp)=(a/2)$
which is an isomorphism at $P^\sharp$,
whence $\ell$-bases $y_1^2y_4$ of $\pgr^2(\OOO_E,J)$ 
and $y_4^2$ of $\pgr^4(\OOO_E,J)$ at $P$.
By Corollary \ref{surjection-Fn-to-grnE}
we can lift $\lambda y_1^2y_4+\mu y_4^2+\cdots\in \pgr^{2,5}(\OOO_E,J)$
to $\bar \beta \in \h^0(J\OOO_E)$, and obtain $\beta$ similarly to the case $b=2$.
\end{proof}

\begin{thm}{\bf Lemma.}\label{H-is-normal}
Under the hypothesis of Lemma \xref{case-up-1-s-in--KX},
assume that $\gr^{2,0}(\OOO,J)\otimes \CC_P$
is generated by $\eta=y_1y_3$ \textup(resp. $y_1^2y_4$\textup)
when $b=2$ \textup(resp. $1$\textup). Then
$\eta\in \beta$ and 
a general member $H$ of $|\OOO_X|_C$ is normal.
\end{thm}

\begin{proof}
The image of $\beta$ under the homomorphism 
\[
\pgr^2(\OOO_E, J) \longrightarrow \gr^{2,0}(\OOO,J) 
\hookrightarrow \gr^1 _C\OOO
\]
in not zero.
\end{proof}

The following lemma will be used often.
\begin{lemma}\label{Lemma-fiber}
Let $(X,C)$ be a extremal curve germ of type \type{(IIA)}.
Let $\KKK\subset I$ be an $I$-primary ideal
and let $\corank(\KKK)$ be the corank of $\KKK$,
that is, the rank of $\OOO/\KKK$ at a general point of $C$.
\begin{enumerate}
\item
If $\corank(\KKK)\le 7$, then $\h^1(\omega/\omega\totimes \KKK)=0$.
\item
If $\corank(\KKK)\le 6$ and $\chi(\omega/\omega\totimes \KKK)= 0$,
then the quotient $\KKK/\Sat(I\KKK)$ has no $\ell$-direct summands
$\SSS$
of the form $(-1)$ or $(i+jP^\sharp)$ with $i\le -2$, $0\le j\le 3$.
\end{enumerate}
\end{lemma}

\begin{proof}
We prove (ii) ((i) is treated similarly).
Take the ideal $\LLL\subset \KKK$ such that $\KKK/\LLL=\SSS$.
Then from the exact sequence
\[
0\longrightarrow \omega\totimes\SSS
\longrightarrow \omega\totimes (\OOO /\LLL)
\longrightarrow \omega\totimes (\OOO /\KKK) \longrightarrow 0
\]
we see that $\chi(\omega\totimes (\OOO /\LLL)) =
\chi( \omega \totimes (\OOO/\KKK)) +\chi( \omega\totimes\SSS)<0$.
If $f$ is birational, we get a contradiction by \cite[1.2.1]{Mori-1988}.
Assume that $f$ is a $\QQ$-conic bundle.
Let $V:=\Spec(\OOO/\LLL)$. By \cite[Th. 4.4]{Mori-Prokhorov-2008}
there is an inclusion $f^{-1}(o)\subset V$.
Since $V\subset 7C$ (as a cycle), we have
\[
2=-K_X\cdot f^{-1}(o)\le -K_X\cdot V\le -7K_X\cdot C=7/4
\]
\cite[Lemma 2.8]{Mori-Prokhorov-2008}, a contradiction.
\end{proof}

\begin{lemma}\label{label-surfaces-C}
Let $S$ be a normal surface and let $C\subset S$ be a
smooth proper curve such that $K_S\cdot C+C^2<0$.
Let $P_1,\dots, P_l\in S$ be all the singular points lying on $C$.
Then $l\le 3$. Moreover, if $l=3$, then the pair $(S,C)$ is plt.
If $l=2$, then $(S,C)$ is plt at least at one of the points $P_1$ or $P_2$.
\end{lemma}

\begin{proof}
Write
\[
(K_S+C)|_C=K_C+\Diff_C(0),\qquad\Diff_C(0)=\sum\delta_i P_i,
\]
where $\Diff_C(0)$ is the \textit{different}, a naturally defined effective $\QQ$-divisor
on $C$ \cite[ch. 16]{Utah}. Then
\[
\sum \delta_i=\deg\Diff_C(0)=2-2p_a(C)+K_S\cdot C+C^2< 2.
\]
Here $\delta_i\ge 1/2$, and $\delta_i< 1$ if and only if $(S,C)$ is plt at $P_i$
\cite[ch. 16, Th. 17.6]{Utah}.
This immediately implies that $l\le 3$. If $l=3$, then $\delta_i<1$ and so $(S,C)$ is plt.
If $l=2$, then either $\delta_1<1$ or $\delta_2<1$.
\end{proof}

\begin{sremark}
Recall that if in the above notation $(S,C)$ is plt at $P_i$, then
$S\ni P_i$ is a cyclic quotient singularity
(see e.g. \cite[ch. 3]{Utah}).
Suppose that $S$ is embedded to a terminal threefold $X$ so that
either $S$ is a Cartier divisor or $S\sim -K_X$.
If $(S,C)$ is plt at $P_i$, then the point $P_i\in X$ is of type \type{cA/n}
(cf. \cite[6.4B]{Reid-YPG1987}, \cite[Prop. 16.17]{Utah}).
\end{sremark}

\section{Case $\ell(P)=1$ and $\mathrm{Sing}(X)=\{P\}$.}\label{section-lP=1-Sing=P}
\begin{de}
In this section additionally to \ref{Set-up} we assume that $\ell(P)=1$ and $\Sing(X)= \{P\}$.
Then by \ref{equation-iP} and \eqref{equation-alpha} we have $i_P(1)=1$,
$\deg\gr_C^1\OOO=0$, and
\[
\alpha\equiv y_1y_2\mod (y_2, y_3, y_4)^2.
\]
In particular, $\h^0(\gr_C^1\OOO)\neq 0$.
\end{de}

\begin{lemma}\label{label-simple-case-1}
$\gr_C^1\OOO\not \simeq\OOO\oplus\OOO$.
\end{lemma}

\begin{proof}
In this case, $y_2\in I^{(2)}$ and
the elements $y_3$, $y_4$ form an $\ell$-free $\ell$-basis of $\gr_C^1\OOO$.
So, after possible change of coordinates,
we have an $\ell$-isomorphism
\[
\gr_C^1\OOO= (P^\sharp)\toplus (2P^\sharp),
\]
where $y_3$ (resp. $y_4$) is an $\ell$-free $\ell$-basis of $(P^\sharp)$
(resp. $(2P^\sharp)$) at $P$.
Thus the conditions of \cite[7.2.1]{Kollar-Mori-1992} are satisfied.
By \cite[7.2.4]{Kollar-Mori-1992} the germ $(X,C)$ is flipping.
This contradicts our assumption in \ref{Set-up}.
\end{proof}

\begin{de}\label{(7.8)}
Thus $\gr_C^1\OOO\not\simeq\OOO\oplus\OOO$.
Since $\h^1(\gr_C^1\OOO)=0$, we have
$\gr^1_C\OOO\simeq\OOO(1)\oplus\OOO(-1)$ (as a sheaf).
So, as in \ref{label-simple-case-1},
the elements $y_3$, $y_4$ form an $\ell$-free $\ell$-basis of $\gr_C^1\OOO$.
We have two possibilities:
\begin{equation}
\label{equation-possibilities-1-gr1COOO}
\gr^1_C\OOO=
\left\{
\begin{array}{lll@{\hbox{\hspace{70pt}}}l}
(1+2P^\sharp)&\toplus&(-1+P^\sharp),
&\flabel{equation-possibilities-1-gr1COOO}{a}
\\
(1+P^\sharp)&\toplus&(-1+2P^\sharp).
&\flabel{equation-possibilities-1-gr1COOO}{b}
\end{array}
\right.
\end{equation}
\end{de}

\begin{de}\label{Subcase(7.8.1)} {\bf Case \fref{equation-possibilities-1-gr1COOO}{a}.}
So we assume that there is an $\ell$-isomorphism
\begin{equation}\label{equation(7.8.1.1)}
\gr^1_C\OOO= (1+2P^\sharp)\toplus(-1+P^\sharp)
\end{equation}
with $\ell$-free $\ell$-bases
$y_4$ and $y_3$ of
$(1+2P^\sharp)$ and $(-1+P^\sharp)$, respectively.
Let $J$ be the ideal such that $I\supset J\supset I^{(2)}$ and
$J/I^{(2)}=(1+2P^\sharp)$.
Then $J^\sharp=(y_4,y_2,y^2_3)$. Since $y^2_3$ must
appear in $\alpha$ by the
description of \type{(IIA)} points, we may assume that 
\[
\alpha\equiv y^2_3+y_1y_2\mod I^\sharp J^\sharp
\]
by replacing $y_3$ with $\lambda y_3$ $(\lambda\in\CC{^*}$). Thus $(y_3,y_4,y_2)$ is a
$(1,2,2)$-monomializing $\ell$-basis of $I\supset J$ at $P$ of the
second kind and
\begin{equation}\label{equation(7.8.1.1)complete-intersection}
J^\sharp=(y_2,y_4).
\end{equation}
In particular, $J^\sharp$ is a local complete intersection.
Consider the exact sequence \eqref{equation-exact-sequence}.
By \cite[ (8.10)]{Mori-1988}, we see $\ell$-isomorphisms
\begin{eqnarray}
\label{eqnarray-new-1}
\gr^1(\OOO,J)&=& (-1+P^\sharp),
\quad
\gr^{2,0}(\OOO,J)= (1+2P^\sharp),
\\
\qquad\gr^{2,1}(\OOO,J)&\simeq&
\gr^1(\OOO,J)^{\totimes 2}\totimes (P^\sharp)
= (-2+3P^\sharp).
\end{eqnarray}
Since $\deg\gr_C^1\OOO=0$ and $\h^1(\gr_C^1\OOO)=0$, we have
an isomorphism of sheaves
\[
\gr^2(\OOO,J)\simeq
\OOO_{\PP^1}(-1)\oplus\OOO_{\PP^1}
\quad\text{or}\quad
\OOO_{\PP^1}(-2)\oplus\OOO_{\PP^1}(1).
\]
Hence, we have the following possibilities for the $\ell$-structure:

\begin{equation}
\label{equation-possibilities-(7.8.1.1)-a}
\gr^2(\OOO,J)=
\left\{\begin{array}{lll@{\hbox{\hspace{30pt}}}l}
(2P^\sharp)&\toplus& (-1+3P^\sharp),&\flabel{equation-possibilities-(7.8.1.1)-a}{a}
\\
(3P^\sharp)&\toplus& (-1+2P^\sharp),&\flabel{equation-possibilities-(7.8.1.1)-a}{b}
\\
(1+2P^\sharp)&\toplus& (-2+3P^\sharp),&\flabel{equation-possibilities-(7.8.1.1)-a}{c}
\\
(1+3P^\sharp)&\toplus& (-2+2P^\sharp).&\flabel{equation-possibilities-(7.8.1.1)-a}{d}
\end{array} \right.
\end{equation}
We consider these cases below in \ref{case-up-1}, \ref{case-up-1b}, \ref{Lemma-possibilities-ac}.

\begin{sde}{\bf Subcase \fref{equation-possibilities-(7.8.1.1)-a}{a}.}
\label{case-up-1}
Thus $\gr^2(\OOO,J)= (2P^\sharp)\toplus (-1+3P^\sharp)$
and $\gr^1_C\OOO= (1+2P^\sharp)\toplus(-1+P^\sharp)$.
We apply Lemma \ref{H-is-normal} with $b=1$, $\eta=y_1^2y_4$
and obtain a global section $\beta_1\in \h^0(J)$ 
with $y_1^2y_4\in \beta_1$.
Moreover, $H$ is normal.
Also by Lemma \ref{lemma-extend-sections-D} there are
global sections $\beta_2,\, \beta_3\in \h^{0}(\OOO_X)$ such that
$\beta_2\equiv y^2_4$, $\beta_3\equiv y_2y_3\mod (y_1)$.
Now we study $(H,P)$. The equation of $H$
satisfies \cite[(7.7.2)]{Kollar-Mori-1992},
because of the sections $\beta_1$, $\beta_2$ and $\beta_3$.
Hence $\Delta(H,C)$ near $P$ has the form
\begin{equation*}
\xy
\xymatrix@R=1pt@C=17pt{
&&\overset3{\circ}\ar@{-}[d]
\\
\bullet\ar@{-}[r]&\underset3\circ\ar@{-}[r]
&\circ\ar@{-}[r]&{\circ}\ar@{-}[r]&\underset4\circ
}
\endxy
\end{equation*}
By Lemma \ref{label-surfaces-C} there is at most
one singular point outside of $P$ and
this point is of type \type{A}.
On the other hand, since $H$ must be contractible either to a Du Val point or to a curve
(see \ref{sde1}), we can attach to the black vertex
exactly one vertex corresponding to a $(-2)$-curve.
So we get \ref{main-theorem-1.1.2}. This
completes \ref{case-up-1}.
\end{sde}

\begin{slemma}\label{case-up-1b}
The subcase \fref{equation-possibilities-(7.8.1.1)-a}{b} does not occur.
\end{slemma}

\begin{proof}
Assume that
\[
\gr^1_C\OOO= (1+2P^\sharp)\toplus(-1+P^\sharp),\quad
\gr^2(\OOO,J)= (3P^\sharp)\toplus (-1+2P^\sharp).
\]
We can choose $\ell$-free $\ell$-bases of
$(3P^\sharp)$ and $(-1+2P^\sharp)$ in the form
$y_2+\cdots$ and $y_4$, respectively.
Recall that $\gr^1(\OOO,J)=I/\Sat_{\OOO}(J)= (-1+P^\sharp)$.
By \ref{equation(7.8.1.1)complete-intersection}
we also have the following $\ell$-isomorphisms
\[
\begin{array}{llll}
\gr^3(\OOO,J)\simeq
\quad\gr^1(\OOO,J)\totimes \gr_C^2(\OOO,J)&=& (0) \toplus (-2+3P^\sharp),
\\[5pt]
\gr^4(\OOO,J)\simeq
\quad\tilde S^2\gr^2(\OOO,J)&=& (1+2P^\sharp) \toplus (P^\sharp)\toplus (-1).
\end{array}
\]
{}From this, one can see that
$\chi(\gr^n(\omega,J))=0$,\, $1$,\, $-1$ for $n=1$,\, $2$, $3$,
respectively.
Now, using exact sequences
\begin{equation}\label{equation-exact-sequence-F-omega}
0\longrightarrow \gr^n(\omega,J)\longrightarrow \omega /F^{n+1}(\omega,J)
\longrightarrow \omega/ F^n(\omega,J)\longrightarrow 0,
\end{equation}
we get $\chi (\omega/ F^n(\omega,J))=0$, $0$, $1$, $0$
for $n=1$, $2$, $3$, $4$, respectively.
Now we apply Lemma \ref{Lemma-fiber}(ii) with $\KKK=F^4(\OOO,J)$
of corank $6$
and get a contradiction.
\end{proof}

\begin{slemma}\label{Lemma-possibilities-ac}
Subcases \fref{equation-possibilities-(7.8.1.1)-a}{c}
and \fref{equation-possibilities-(7.8.1.1)-a}{d} do not occur.
\end{slemma}

\begin{proof}
Apply Lemma \ref{Lemma-fiber}(ii) with $\KKK=J$ using
\eqref{eqnarray-new-1}.
\end{proof}
This completes our treatment of all the
possibilities in \eqref{equation-possibilities-(7.8.1.1)-a}.
\end{de}

\begin{de}\label{Case(7.8.2)} {\bf Case \fref{equation-possibilities-1-gr1COOO}{b}.}
We have an $\ell$-isomorphism
\begin{equation}\label{equation-(7.3.1.1)}
\gr^1_C\OOO= (1+P^\sharp)\toplus(-1+2P^\sharp)
\end{equation}
with $\ell$-free $\ell$-bases $y_3$ and $y_4$ of
$(1+P^\sharp)$ and $(-1+2P^\sharp)$, respectively.
Let $J$ be the laminal ideal of width $2$ with $J/I^{(2)}=(1+P^\sharp)$.
We also have $\ell$-isomorphisms
\[
\gr^1(\OOO,J)=
(-1+2P^\sharp),\qquad\gr^{2,0}(\OOO,J)= (1+P^\sharp).
\]
Note that $y^2_4$ does not
appear in $\alpha$ because $\wt y^2_4\not\equiv\wt\alpha$.
Since $I^\sharp=(y_2,y_3,y_4)$ and $J^\sharp=(y_2,y_3,y^2_4)$,
we may further assume $\alpha\equiv y_1y_2\mod I^\sharp J^\sharp$ after
changing coordinates $y_2 \longmapsto y_2+\lambda y_1y_4^2$.
Hence, $y_2\in F^3(\OOO,J)$ and we have an
$\ell$-isomorphism
\[
\gr^{2,1}(\OOO,J)\simeq\gr^1(\OOO,J)^{\totimes 2}= (-1),
\]
where $(-1)$ has an $\ell$-free $\ell$-basis $y^2_4$
and $\gr^{2,0}(\OOO,J)= (1+P^\sharp)$.
Further, we have
\begin{eqnarray*}
\chi(\omega/\omega\totimes J)&=&
\chi(\gr^0(\omega,J))+\chi(\gr^1(\omega,J))
\\
&=&
\chi(\gr(\OOO,J)\totimes\omega)+\chi(\gr(\OOO,J)\totimes\omega)
\\
&=&\chi(-1+3P^\sharp)+\chi(-1+P^\sharp)=0.
\end{eqnarray*}
Then using \eqref{equation-exact-sequence}
and Lemma \ref{Lemma-fiber}(ii)
we get the following possibilities:
\begin{equation}\label{equation-division-cases-(7.8.2.2)}
\gr^2(\OOO,J)=
\left\{\begin{array}{lll@{\hbox{\hspace{70pt}}}l}
(P^\sharp)&\toplus& (0),&\flabel{equation-division-cases-(7.8.2.2)}{a}
\\
(1)&\toplus& (-1+P^\sharp).&\flabel{equation-division-cases-(7.8.2.2)}{b}
\end{array} \right.
\end{equation}
In the situation \fref{equation-division-cases-(7.8.2.2)}{a},
by \cite[(7.3.4)]{Kollar-Mori-1992} the germ $(X,C)$ is flipping.
Below we consider the possibility \fref{equation-division-cases-(7.8.2.2)}{b}.

\begin{sde}{\bf Subcase \fref{equation-division-cases-(7.8.2.2)}{b}.}
\label{pf-(7.8.2.2)}
Then
\[
\gr^1_C\OOO= (1+P^\sharp)\toplus(-1+2P^\sharp),\quad
\gr^2(\OOO,J)= (1)\toplus (-1+P^\sharp).
\]
Let $D\in |-K_X|$ be as in \ref{ge}.
Because of \eqref{equation-exact-sequence}, the
$\ell$-summand $(1)\subset \gr^2(\OOO,J)$ is generated by an element $u\in\OOO_{X,P}$ such
that $u\equiv y^2_4+y_3(y_1+\cdots) \mod F^3(\OOO,J)$ after replacing $y_3$ with
$\lambda y_3$ for some $\lambda\in\CC{^*}$.
By Lemma \ref{lemma-extend-sections-D}
there is a section $\beta_1\in \h^{0}(\OOO_X)$ such that
$\beta_1\equiv u\mod (y_1)+(y_2,\, y_3,\, y_4)^3$. Since $\beta_1(P)=0$, we see
$\beta_1\in \h^{0}(I)\simeq \h^{0}(J)$ by
$\h^{0}(\gr^1(\OOO,J))=0$.
Thus $\beta_1$ induces a
section $(\unit)\cdot u$ of $(1)\subset \gr^2(\OOO,J)$ at $P$.
In particular, $y_1y_3$ appears in $\beta_1$.
By Lemma \ref{lemma-extend-sections-D} terms $y^2_4$ and $y_2y_3$
also appear in the equation of general $H\in |\OOO_X|_C$.
Let $\beta\in \h^{0}(I)$ be a general section.
Then, by the above, $\beta\in \h^{0}(J)$
and the induced section $\overline{\beta}$ of $\gr^2(\OOO,J)$ is a basis
of the $\ell$-summand $(1)\subset \gr^2(\OOO,J)$. Thus its image in $\gr^1_C\OOO$ has
exactly one simple zero outside of $P$. Hence $H=\{\beta=0\}\in |\OOO_X|_C$ has
exactly one singular point outside of $P$, say $R$.
As for $(H,P)$, we can apply \cite[(7.7.1)]{Kollar-Mori-1992} by the above.
So, $\Delta(H,C)$ near $P$ has the form
\begin{equation*}
\xy
\xymatrix@R=2pt@C=25pt{
&\circ\ar@{-}[d]
\\
\underset4\circ\ar@{-}[r]
&\circ\ar@{-}[r]&\underset4{\circ}\ar@{-}[r]&\bullet
}
\endxy
\end{equation*}
By Lemma \ref{label-surfaces-C} the point
$H\ni R$ is of type \type{A}.
Since $C$ is either contractible to a Du Val point or
it is a fiber of a rational curve fibration (see \ref{sde1}), it is easy to see
that $H\ni R$ is an \type{A_2} point and $\Delta (H, C)$ is as in \ref{main-theorem-1.1.1}.
This completes \ref{pf-(7.8.2.2)} and our treatment of \ref{Case(7.8.2)}.
\end{sde}
\end{de}

\section{Case $\ell(P)=1$ and $\mathrm{Sing}(X)\neq \{P\}$.}\label{section-lP=1-Sing-ne=P}
\begin{de}\label{(7.8)a}
In this section we assume that $\ell(P)=1$
and $\Sing (X)\neq \{ P\}$.
Then $X$ has exactly
one singular point outside of $P$, which is a type \type{(III)} point, say $R$
\cite[6.2]{Mori-1988}, \cite[9.1]{Mori-Prokhorov-2008}. Moreover, $i_{R}(1)=1$
(otherwise we can apply
deformations \cite[4.7]{Mori-1988} to get an extremal curve germ having
a type \type{(IIA)} point and more than one
type \type{(III)} points which is impossible).
By \cite[2.3.2]{Mori-1988} we have
$\deg\gr_C^1\OOO = -i_P(1)=-1$.
Since $\h^1(\gr_C^1\OOO)=0$, there exists
an isomorphism of sheaves
$\gr^1_C\OOO\simeq\OOO\oplus\OOO(-1)$.
In particular, $\h^0(\gr_C^1\OOO)\neq 0$.
As in \ref{label-simple-case-1} and \ref{(7.8)} we see that
the elements $y_3$, $y_4$ form an $\ell$-free $\ell$-basis of $\gr_C^1\OOO$.
So,
\begin{equation}\label{cases-gr1COOO-S4}
\gr^1_C\OOO=
\left\{
\begin{array}{lll@{\hbox{\hspace{70pt}}}l}
(P^\sharp)&\toplus&(-1+2P^\sharp),&\flabel{cases-gr1COOO-S4}{a}
\\
(2P^\sharp)&\toplus&(-1+P^\sharp).&\flabel{cases-gr1COOO-S4}{b}
\end{array}
\right.
\end{equation}
For $R$, we write
\[
(X,R)=\{\gamma=0\}\subset\CC^4_{z_1,\dots,z_4}\ \supset\ (C,R)=\{z_1\text{-axis}\},
\]
where $\gamma=\gamma(z_1,\dots,z_4)$ is such that $\gamma\equiv z_1z_2\mod (z_2,z_3,z_4)^2$.
\end{de}

\begin{de}{\bf Case \fref{cases-gr1COOO-S4}{a}.}\label{case-(7.3.1.2)a}
Thus
\begin{equation}\label{equation-(7.3.1.1)a}
\gr^1_C\OOO= (P^\sharp)\toplus(-1+2P^\sharp).
\end{equation}
Let $J$ be the laminal ideal of width $2$ with $J/I^{(2)}=(P^\sharp)$
in the above decomposition.
Without loss of generality, in the decomposition \eqref{equation-(7.3.1.1)a},
we may assume that $y_3$ and $y_4$
form $\ell$-free $\ell$-bases of $(P^\sharp)$
and $(-1+2P^\sharp)$ at $P$ and
$z_3$ and $z_4$ form free bases of $(P^\sharp)$
and $(-1+2P^\sharp)$ at $R$.
If there is an $\ell$-isomorphism
\begin{equation}\label{equation-(7.3.1.2)a}
\gr^2(\OOO,J)= (P^\sharp)\toplus(0),
\end{equation}
then we can apply \cite[7.3.4]{Kollar-Mori-1992}
and see that the germ $(X,C)$ is flipping, a contradiction.
Thus we assume that
$\gr^2(\OOO,J)\neq (P^\sharp)\toplus(0)$.
Note that $J$ satisfies
\[
\gr^1(\OOO,J)= (-1+2P^\sharp),\qquad
\gr^{2,0}(\OOO,J)= (P^\sharp).
\]
We also have the equality at $P$:
\[
\gr^{2,1}(\OOO,J)=\gr^1(\OOO,J)^{\totimes 2}
\]
which is proved by the same argument as
the one after \eqref{equation-(7.3.1.1)}.
If $z^2_4$ appears (resp. does not appear)
in $\gamma$, then we have the equality at $R$:
\[
\gr^{2,1}(\OOO,J)=\gr^1(\OOO,J)^{\totimes 2}(R)\qquad
\text{(resp. $\gr^1(\OOO,J)^{\totimes 2}$). }
\]

First we assume that $z^2_4$ appears in $\gamma$. In this case, we have
$\ell$-isomorphisms $\gr^{2,1}(\OOO,J)= (0)$,
$\gr^{2,0}(\OOO,J)= (P^\sharp)$,
and the $\ell$-exact sequence
\eqref{equation-exact-sequence}
is $\ell$-split. Whence we get \eqref{equation-(7.3.1.2)a}, a contradiction.

Hence $z^2_4$ does not
appear in $\gamma$, that is, $\gr^{2,1}(\OOO,J)=\gr^1(\OOO,J)^{\totimes 2}$ at $R$. In this case,
we have an $\ell$-isomorphism $\gr^{2,1}(\OOO,J)= (-1)$.
If $\gr^2(\OOO,J)= (P^\sharp)\toplus(-1)$, then
we apply Lemma \ref{Lemma-fiber}(ii) with $\KKK=J$ of corank $2$
and get a contradiction.
Therefore,
\begin{equation*}
\gr^2(\OOO,J)= (0)\toplus (-1+P^\sharp).
\end{equation*}
\begin{sde}
\label{case-(7.3.1.2)a-new}
Similarly to \ref{pf-(7.8.2.2)} we can prove that a general section $\beta\in \h^{0}(I)$ defines a
surface $H$ which is smooth outside of $\{P,R\}$ and the term $y_1y_3$
appears in $\beta$ at $P$.
By Lemma \ref{lemma-extend-sections-D} terms $y^2_4$ and $y_2y_3$
also appear in $\beta$ at $P$.
Thus we can apply \cite[7.7.1]{Kollar-Mori-1992}.
By Lemma \ref{label-surfaces-C} the point
$(H,R)$ is of type \type{A}.
The rest of the arguments are the same
as \ref{pf-(7.8.2.2)}.
We get \ref{main-theorem-1.1.1}.
This completes our treatment of \ref{case-(7.3.1.2)a}.
\end{sde}
\end{de}

\begin{de}{\bf Case \fref{cases-gr1COOO-S4}{b}.}
Thus
\[
\gr^1_C\OOO= (2P^\sharp)\toplus (-1+P^\sharp),
\]
where $y_3$ and $y_4$ form $\ell$-free $\ell$-bases of $(-1+P^\sharp)$ and $(2P^\sharp)$ at $P$.

Let $J$ be the ideal such that $I\supset J\supset I^{(2)}$ and
$J/I^{(2)}=(2P^\sharp)$.
Then $J^\sharp=(y_4,\, y_2,\, y^2_3)$. Since $y^2_3$ must
appear in $\alpha$ by the
description of \type{(IIA)} points, we may assume
\[
\alpha\equiv y^2_3+y_1y_2\mod I^\sharp J^\sharp.
\]
Thus $(y_3,\, y_4,\, y_2)$ is a
$(1,\, 2,\, 2)$-monomializing $\ell$-basis of $I\supset J$ at $P$ of the
second kind \cite[8.11]{Mori-1988}.
By changing coordinates $z_1, z_2$ at $R$, we can further
assume that
$\gamma=z_1z_2-\phi(z_3,z_4)$ with $\phi\in (z_3,z_4)^2$
so that
$z_3$ and $z_4$ form at $R$
free bases of $(-1+P^\sharp)$ and $(2P^\sharp)$ given above, respectively.
Then
$(z_3,z_4,z_2)$ is $(1,2,b)$-monomializing $\ell$-basis of the second kind,
where $b=\ord_{(1,2)}(\phi) \ge 2$.
Now \cite[(8.12)(ii)]{Mori-1988} applies to our case with
$d=2$, $t=s=2$, $s'=0$, $P_1=P$, $P_2=R$, $b_1=2$, $b_2=b$,
$\MMM=\gr^1(\OOO, J)=(-1+P^\sharp)$,
$\LLL=\gr^{2,0}(\OOO,J)=(2P^\sharp)$,
$\DDD_1 = (P^\sharp)$, $\DDD_2 = (R)=(1)$. We obtain
\begin{equation*}
\begin{array}{llll}
\gr^{2,0}(\OOO,J)=2P^\sharp, &&
\gr^{2,1}(\OOO,J)
=-2+ \lfloor{2/b} \rfloor+3P^\sharp,
\\[4pt]
\gr^{3,0}(\OOO,J)
=(-1+3P^\sharp), &&
\gr^{3,1}(\OOO,J)
=-2+\lfloor{3/b} \rfloor.
\end{array}
\end{equation*}

By the $\ell$-exact sequence
\eqref{equation-exact-sequence}
and $\h^1(\gr^2(\OOO,J))=0$, we have one of the following
possibilities for the $\ell$-structure:
\begin{equation}\label{equation-possibilities-(7.8.1.1)}
\gr^2(\OOO,J)=
\left\{
\begin{array}{ll@{\hbox{\hspace{9pt}}}l}
(-1+2P^\sharp)\toplus (-1+3P^\sharp) &\text{if $b \ge 3$},& \flabel{equation-possibilities-(7.8.1.1)}{a}
\\
(2P^\sharp)\toplus (-1+3P^\sharp) &\text{if $b=2$}. & \flabel{equation-possibilities-(7.8.1.1)}{b}
\end{array} \right.
\end{equation}

\begin{slemma}
The subcase \fref{equation-possibilities-(7.8.1.1)}{a} does not occur.
\end{slemma}

\begin{proof}
Note that $F^3(\OOO,J)$ is of corank 4.
Since $b \ge 3$, we see
\[
\chi(\omega/F^3(\omega,J))=
\sum_{i=0,1} \chi(\gr^i(\omega,J))+\sum_{i=0,1}
\chi(\gr^{2,i}(\omega,J))=0.
\]
Thus we see that $\gr^3(\OOO,J) \simeq \OOO_C(-1)\oplus\OOO_C(-2)$
or $\OOO_C(-1)^{\oplus 2}$
as a sheaf, and
$\gr^3(\OOO,J)$ must have an $\ell$-direct summand which equals
$(-1)$ or lower, whence a contradiction by Lemma 2.6(ii).
\end{proof}

\begin{sde}{\bf Subcase \fref{equation-possibilities-(7.8.1.1)}{b}.}
\label{Subcasegr2OOOJsimeq2Ptoplus-1+3P}
Thus
\[
\gr^1_C\OOO= (2P^\sharp)\toplus (-1+P^\sharp), \quad
\gr^2(\OOO,J)=(2P^\sharp)\toplus (-1+3P^\sharp).
\]
This possibility can be treated similarly to \ref{case-up-1} and
we get \ref{main-theorem-1.1.2} (and \ref{1-main-theorem-1.1.2}).
\end{sde}
\end{de}

\begin{thm}{\bf Proposition.}
\label{proposition-existence}
All the possibilities
\fref{equation-division-cases-(7.8.2.2)}{b}, 
\fref{cases-gr1COOO-S4}{a},
\fref{equation-possibilities-(7.8.1.1)-a}{a}, and
\fref{equation-possibilities-(7.8.1.1)}{b} do occur.
\end{thm}

\begin{proof}
We use deformation arguments (cf. \cite[\S 11]{Kollar-Mori-1992}). 
Suppose we are given $\alpha,\beta\in \CC\{y_1,\dots, y_4\}$ satisfying 
\ref{case-up-1},
\ref{case-(7.3.1.2)a-new}, \ref{Subcasegr2OOOJsimeq2Ptoplus-1+3P}, or \ref{pf-(7.8.2.2)}
and the corresponding $H \supset C$ as in \xref{main-theorem-1.1.1} or \xref{main-theorem-1.1.2}. 
Consider a small deformation $H_t=\{\alpha=\beta-t=0\}/\muu_4$
of $P\in H\subset U_P \subset X$ near the \type{cAx/4}-point $P$ 
(see
\cite[(7.7.1)]{Kollar-Mori-1992} for the case \ref{main-theorem-1.1.1} and 
\cite[(7.7.2)]{Kollar-Mori-1992} for the case \ref{main-theorem-1.1.2}).
In cases \fref{cases-gr1COOO-S4}{a} and \fref{equation-possibilities-(7.8.1.1)}{b}
we similarly construct a small deformation $H_t$ of $R\in H\subset U_R \subset X$ 
also near the \type{(III)}-point $R$.
Further, by the arguments similar to \cite[11.4.2]{Kollar-Mori-1992} we see that 
the natural morphism $\operatorname{Def} H \to \operatorname{Def} (H,P)$
(resp. $\operatorname{Def} H \to \operatorname{Def} (H,P)\prod \operatorname{Def} (H,R)$)
is smooth. 
Then we construct a threefold $X$ as a total one-parameter deformation space
$X=\cup H_t$ which induces a local deformation of $H$ in $U_P$ (resp., and $U_R$). 
This shows the existence of $X\supset C$
with $H\in |\OOO_X|_C$ and such that $C\cap U_P$ (resp. $C\cap U_P$ and $C\cap U_R$)
has the desired structure. (Note however that we do not assert that $H$ is general in $|\OOO_X|_C$.) 
By the construction and by \eqref{equation-alpha} we have $\ell(P)=1$.
The contraction $f: X\to Z$ exists by \cite[11.4.1]{Kollar-Mori-1992}
and it is divisorial because $C\subset H$ can be contracted to a Du Val point. 
It remains to show that the constructed $X$ has the desired 
sheaf $\gr_C^1\OOO$:
\smallskip
{\rm \begin{center}
\begin{tabular}{l|c|c}

Ref. (case division and conclusion) & $\gr_C^1\OOO$&$\Delta(H,C)$ 
\\[7pt]
\hline
&&
\\[-5pt]
\fref{equation-possibilities-1-gr1COOO}{b}-\fref{equation-division-cases-(7.8.2.2)}{b}, \ref{pf-(7.8.2.2)}
&$(1+P^\sharp)\toplus(-1+2P^\sharp)$&\xref{main-theorem-1.1.1}

\\
\fref{cases-gr1COOO-S4}{a}, \ref{case-(7.3.1.2)a-new} 
&$(P^\sharp)\toplus(-1+2P^\sharp)$&\xref{main-theorem-1.1.1}

\\
\fref{equation-possibilities-1-gr1COOO}{a}-\fref{equation-possibilities-(7.8.1.1)-a}{a}, \ref{case-up-1}
&$(1+2P^\sharp)\toplus(-1+P^\sharp)$&
\xref{main-theorem-1.1.2}
\\

\fref{cases-gr1COOO-S4}{b}-\fref{equation-possibilities-(7.8.1.1)}{b}, \ref{Subcasegr2OOOJsimeq2Ptoplus-1+3P}
&$(2P^\sharp)\toplus(-1+P^\sharp)$&
\xref{main-theorem-1.1.2}
\end{tabular}
\end{center}}
Consider, for example, the case \fref{equation-possibilities-(7.8.1.1)}{b}
(other cases are similar). 
The local equation $\beta$ of $H$ at $P$ gives us a local generator 
of the positive (rank 1) part of $\gr_C^1\OOO$, because
$\beta\equiv (\unit)\cdot y_1^2y_4\mod I_C^{(2)}$.
Therefore, $\gr_C^1\OOO\simeq (1+2P^\sharp)\toplus(-1+P^\sharp)$ or $(2P^\sharp)\toplus(-1+P^\sharp)$.
Since we have an extra \type{(III)}-point, the first possibility does not occur
by \eqref{cases-gr1COOO-S4}.
\end{proof}

\section{Case $\ell(P)=3$.}\label{section-lP=3}

\begin{lemma}\label{lemma-new-ell-p}
Assume that $\h^0(\gr^1_C\OOO)\neq 0$. Then we have the following.
\begin{enumerate}
\item 
If $\ell(P)\ge 3$, then $\Sing(X)=\{P\}$.
\item 
$\ell(P)\le 5$.
\end{enumerate}
\end{lemma}
\begin{proof}
(i)
Assume that $X$ has another singular point $R$.
Since $i_P(1)\ge 2$ (see \eqref{equation-iP-lP}) and $i_R(1)\ge 1$, $\deg\gr_C^1\OOO\le -2$.
This contradicts $\h^0(\gr^1_C\OOO)\neq 0$
(see \ref{notation-grn}).
(ii) follows similarly from $i_P(1)\le 3$
and $\ell(P)\not \equiv 2\mod 4$ (see \ref{(7.5)}).
\end{proof}
Below in this section we assume that $\ell(P)=3$ and $\h^0(\gr^1_C\OOO)\neq 0$.
\begin{de}
\label{(7.10)}
We use the notation of \ref{(7.5)} at $P$.
In particular, the equation of $X$
in $\CC^4_{y_1,\cdots,y_4}/\muu_4(1,1,3,2)$ has the following form
\begin{equation}
\label{equation-alpha-lP=3}
\alpha=y_1^3y_3+y_2^2+y_3^2+\delta y_4^{2k+1}+cy_1^2y_4^2+\epsilon y_1y_3y_4+\cdots,
\end{equation}
where $\delta\neq 0$ and $2k+1$ is the lowest power of $y_4$ appearing in $\alpha$.

\begin{sde}\label{(7.10)-a}
Since $i_P(1)=2$ (see \ref{equation-iP}), we have
$\deg \gr^1_C\OOO=-1$. Hence,
$\gr^1_C\OOO\simeq\OOO\oplus\OOO(-1)$ because $\h^1(\gr^1_C\OOO)=0$.
Moreover, $y_2$ and $y_4$ form an $\ell$-free $\ell$-basis of $\gr^1_C\OOO$ at $P$.
Therefore, we have one of the following $\ell$-isomorphisms
\begin{equation}\label{equation-(7.4.1.1)}
\gr^1_C\OOO=
\left\{\begin{array}{l@{\hbox{\hspace{70pt}}}l}
(2P^\sharp)\toplus(-1+3P^\sharp),&\flabel{equation-(7.4.1.1)}{a}
\\
(3P^\sharp)\toplus(-1+2P^\sharp).&\flabel{equation-(7.4.1.1)}{b}
\end{array} \right.
\end{equation}
\end{sde}
\end{de}

\begin{de}
In the case \fref{equation-(7.4.1.1)}{a} the germ $(X,C)$ is flipping
by \cite[(7.4.4)]{Kollar-Mori-1992}.
So we assume that there is an $\ell$-isomorphism
\begin{equation}\label{equation(7.8.1.1)-b}
\gr^1_C\OOO= (3P^\sharp)\toplus(-1+2P^\sharp).
\end{equation}
In this case, $y_2$ and $y_4$ form $\ell$-free $\ell$-bases of
$(3P^\sharp)$ and $(-1+2P^\sharp)$, respectively.
We investigate a general member $H\in |\OOO_X|_C$.
Let $\beta=0$ be its local equation at $P$.
By \eqref{equation(7.8.1.1)-b}
the term $y_1^2y_4$ does not appear in $\beta$
(otherwise it generates a global section of
$(-1+2P^\sharp)$).
On the other hand, $y^2_4$ and $y_2y_3$ must
appear in $\beta$ by Lemma \ref{lemma-extend-sections-D}.
Write 
\begin{equation}\label{equation-beta-new}
\beta= \theta y_4^2+\nu y_2y_3+\lambda y_1y_3+\mu y_1^3y_2+\cdots,
\ \lambda, \mu, \in \CC,\ \theta, \nu\in \CC^*.
\end{equation}
\end{de}

\begin{de}
Let $J$ be the ideal such that $I\supset J\supset I^{(2)}$ and
$J/I^{(2)}=(3P^\sharp)$.
Then $J^\sharp=(y_2,\, y_3,\, y^2_4)$.
Since $y^2_4$ does
not appear in $\alpha$, we have
\begin{equation}
\label{equation-sect5-new-alpha}
\alpha\equiv y^3_1y_3+y^2_1y^2_4\gamma (y^4_1)=
y_1^2\left(y_1y_3+y^2_4\gamma (y^4_1)\right)
\mod I^\sharp J^\sharp
\end{equation}
for some $\gamma (T)\in\CC\{T\}$ such that $\gamma(0)=c$.
If $c\neq 0$, we get $\gr^{2,1}(\OOO,J)=(-1+P^\sharp)$
with an $\ell$-basis $y_3$.
If $c= 0$, then $\gr^{2,1}(\OOO,J)=(-1)$ with an $\ell$-basis
$y_4^2$.
Then using \eqref{equation-exact-sequence}
we get the following possibilities:
\begin{equation}\label{equation-lP=3-gr2O-1}
\gr^2(\OOO,J)=
\left\{
\begin{array}{lll@{\hbox{\hspace{70pt}}}l}
(P^\sharp)&\toplus&(-1+3P^\sharp),&\flabel{equation-lP=3-gr2O-1}{a}
\\
(3P^\sharp)&\toplus&(-1+P^\sharp),&\flabel{equation-lP=3-gr2O-1}{b}
\\
(0)&\toplus&(-1+3P^\sharp), &\flabel{equation-lP=3-gr2O-1}{c}
\\
(3P^\sharp)&\toplus&(-1),&\flabel{equation-lP=3-gr2O-1}{d}
\end{array}
\right.
\end{equation}
where $c\neq 0$ if and only if we are in the case
\fref{equation-lP=3-gr2O-1}{a} or \fref{equation-lP=3-gr2O-1}{b}.
The case \fref{equation-lP=3-gr2O-1}{d} is disproved 
by applying Lemma \ref{Lemma-fiber}(ii) with $\KKK=J$.
\end{de}

\begin{de}{\bf Subcases \fref{equation-lP=3-gr2O-1}{a} and \fref{equation-lP=3-gr2O-1}{c}.}
\label{Subcases-lP=3-34-1}
We use the arguments \ref{definition-J}-\ref{y12y4-in-beta}.
Since $\gr^1(\OOO,J)= (-1+2P^\sharp)$
and $\gr^2(\OOO,J)= (aP^\sharp)\toplus(-1+3P^\sharp)$ with $a=1$ or $0$,
we can apply Lemma \ref{case-up-1-s-in--KX} and get a member 
$E\in |-K_X|$ and an $\ell$-isomorphism 
\[
\gr^2(\OOO,J)=(aP^\sharp)\toplus \OOO_C(E).
\]
Let 
$pF^n(\OOO_E,J)$ and $\pgr^n(\OOO_E,J)$ be as in \ref{define-pF}.
Then by \eqref{equation-new-pgr2}
and Corollary \ref{y12y4-in-beta} there exists $\beta\in \h^0(J)$ 
as in \eqref{equation-beta-new} such that $\lambda$ and $\theta$ 
are independent over the coefficients of $\alpha$.
The image of $\beta$ in $\gr^1_C\OOO$
is a global generator of $(3P^\sharp) \subset \gr^1_C\OOO$
and so $y_1^3y_2 \in \beta$, i.e. $\mu\neq 0$.
In particular, $H$ is normal.
Now we can apply Computation \ref{computation-lP=3a-new1}
in which $(a,b)=(1,0)$ and $\lambda$ is general with respect to the coefficients of 
$\alpha$.
By \ref{sde1}, we see that the contraction
$f$ is divisorial and the only case \ref{computation-lP=3a-new1} \type{a)} occurs.
Thus we obtain \ref{main-theorem-1.1.3}.
\end{de}

The existence of the cases \fref{equation-lP=3-gr2O-1}{a} and \fref{equation-lP=3-gr2O-1}{c} 
can be shown similarly to Proposition \ref{proposition-existence}.
This also follows from the following.

\begin{de}{\bf Example.}\label{example-1-1-3}
Let $Z \subset {\mathbb C}^5_{z_1,\ldots,z_5}$ be defined by two equations:
\begin{eqnarray*}
0&=& z_2^2+z_3+z_4z_5^k+z_1^3,\qquad k\ge 1,\\
0&=& z_1^2z_2^2 + z_4^2-z_3z_5+z_1^3z_2+cz_1^2z_4.
\end{eqnarray*}
By eliminating $z_3$ using the first equation,
one sees easily that $(Z,0)$ is a
threefold singularity of type \type{cD_{5}}.
Let $B \subset Z$ be the $z_5$-axis, and
let $f : X \to Z$ be the weighted blowup with
weight $(1,1,4,2,0)$. So the support of the center
of the blowup coincides with $B$.
In the weighted blowup computation one sees easily that
$C:=f^{-1}(0)_{\red} \simeq {\PP}^1$
and $X$ is covered by two charts: $z_1$-chart
and $z_3$-chart. The origin
of the $z_3$-chart
is a type \type{(IIA)} point $P$ with $\ell(P)=3$:
\[
\{y_1^3y_3+y_2^2+y_3^2+y_4(y_1^2y_2^2+y_4^2+y_1^3y_2+c y_1^2 y_4)^k=0\}/\muu_{4}(1,1,3,2),
\]
where $(C,P)$ is the $y_1$-axis.
Moreover, $X$ is smooth outside of $P$.
Thus $X\to Z$ is a divisorial contraction of type 
\fref{equation-lP=3-gr2O-1}{a} if $c\neq 0$ and $k=1$  and \fref{equation-lP=3-gr2O-1}{c} otherwise.
\end{de}

\begin{lemma}\label{Lemma-Subcases-lP=3-34}
The subcase \fref{equation-lP=3-gr2O-1}{b} does not occur.
\end{lemma}

\begin{proof}
Assume the contrary, that is,
\[
\gr^1_C\OOO= (3P^\sharp)\toplus(-1+2P^\sharp),\quad
\gr^2(\OOO,J)= (3P^\sharp)\toplus(-1+P^\sharp).
\]
We have $I^\sharp=(y_2,y_3,y_4)$, the equation $\alpha$ has the form
\[
\alpha=y_1^3y_3+y_2^2+y_3^2+\delta y_4^{2k+1}+c y_1^2y_4^2+\epsilon y_1y_3y_4+\cdots,\quad c\neq 0,
\]
and $J^\sharp=(y_2,y_3,y_4^2)$
by the choice of $y_2$ and $y_4$ in \ref{(7.10)-a}.
Therefore, $I^\sharp J^\sharp=(y_2^2, y_2y_3, y_2y_4, y_3^2, y_3y_4, y_4^3)$
and $y_3y_4, y_4^3 \in I^\sharp J^\sharp=F^3(\OOO,J)$.
Hence a general deformation of the form $\alpha_{t_1,t_2}=\alpha+t_1y_4^3+t_2y_1y_3y_4$
preserves our assumptions on the $\ell$-splittings of $\gr^1_C\OOO$ and $\gr^2(\OOO,J)$.
Thus in the proof of Lemma \ref{Lemma-Subcases-lP=3-34},
we may assume that $k=1$ and the coefficients $\epsilon$ and $\delta$ are general enough.

Denote $z:=cy_4^2+y_1y_3$. Then
$y_1^2z=y_1^3y_3+cy_1^2y_4^2\in I^\sharp J^\sharp$. Hence
$z\in \Sat(I^\sharp J^\sharp)$ and
$y_2$, $y_3$ generate the rank-two sheaf $J^\sharp/(z)+I^\sharp J^\sharp$.
Thus $y_2, y_3$ form a free basis of $J^\sharp/(z)+I^\sharp J^\sharp$
and $\Sat(I^\sharp J^\sharp)=(z)+I^\sharp J^\sharp$.
In the natural diagram
\[
\xymatrix@R=13pt{
*+[l]{(-1)\ =\ (1+2P^\sharp)^{\totimes 2}}\ar@/_7pt/[dr]\,\ \ar@{^{(}->}[r] &\gr^{2,1}(\OOO,J)
\\
&(-1+P^\sharp),\ar[u]
}
\]
the $\ell$-invertible sheaf $(-1)$ (resp. $(-1+P^\sharp)$)
has an $\ell$-basis $y_4^2$ (resp. $y_3$)
since $cy_4^2 \equiv y_1y_3 \mod F^3(\OOO,J)=\Sat(I^\sharp J^\sharp)$.
Further, the standard $\ell$-exact sequence (see \eqref{equation-exact-sequence})
\begin{equation}\label{Lsplitting-for-gr2OJ}
0 \longrightarrow (-1+P^\sharp) \longrightarrow
\gr^2(\OOO,J) \longrightarrow (3P^\sharp) \longrightarrow 0
\end{equation}
is $\ell$-split.
Recall that $(3P^\sharp)$ has an $\ell$-basis $y_2$ at $P$.
We can write $\alpha\equiv y_4^3+y_1^2z+ \epsilon y_1y_3y_4 \mod J^{\sharp 2}$.
Recall that $c\neq 0$.
Since $y_4z\in I^\sharp\Sat(I^\sharp J^\sharp)$ and $J\supset I^{(2)}$,
we have $cy_4^3+y_1y_3y_4=y_4z\in J^{\sharp (2)}$.
On the other hand, $y_4^3+y_1^2z+\epsilon y_1y_3y_4\in J^{\sharp 2}$.
So,
\begin{equation}\label{compKatP}
(c\epsilon - 1)y_4^3 \equiv y_1^2 z, \quad
\bigl(\epsilon -c^{-1}\bigr) y_3y_4\equiv -y_1z\mod J^{\sharp (2)}.
\end{equation}
Since
\[
\Sat(I^\sharp J^\sharp)/J^{\sharp 2}=\bigl((z)+y_4 J^\sharp+J^{\sharp 2}\bigr)/J^{\sharp 2},
\]
$\gr^3(\OOO,J)^\sharp=\Sat(I^\sharp J^\sharp)/ J^{\sharp (2)}$ is generated by
$z$ and $y_2y_4$ at $P^\sharp$.

Take the ideal $\KKK$ so that
$J\supset \KKK\supset \Sat(IJ)$ and $\KKK/\Sat(IJ)=(3P^\sharp)$,
with $\ell$-basis $y_2$.
Hence $\KKK^\sharp/J^{\sharp (2)}$ is generated by $y_2, z$ at $P^\sharp$.
The following sequence \cite[(8.6)]{Mori-1988}
\[
0\longrightarrow (-1+2P^\sharp)^{\totimes 3} \longrightarrow
\gr^3(\OOO,\KKK)\longrightarrow(3P^\sharp)\longrightarrow 0
\]
is exact outside of $P$ because,
in suitable local coordinates
$u,\, v,\, w$ at some point $Q\in C$, $Q\neq P$, we have $I=(u,v)$, $J=(u,v^2)$, $\KKK=(u,v^3)$,
and
$\KKK I=(u^2, uv, v^4)\supset J^2=(u^2,uv^2, v^4)$.
We also note $\Sat(\KKK^\sharp I^\sharp)\supset J^{\sharp (2)}$
by the above computation.
The above sequence induces the following $\ell$-exact sequence
\begin{equation}\label{Lsplitting-for-gr3OK}
0\longrightarrow (-1)\longrightarrow
\gr^3(\OOO,\KKK)\longrightarrow(3P^\sharp)\longrightarrow 0
\end{equation}
because
$(-1+2P^\sharp)^{\totimes 3} \hookrightarrow
(-1)\rightarrow\gr^3(\OOO,\KKK)$,
where $z$ is an $\ell$-basis of
$(-1)=\Sat(I^\sharp J^\sharp)/\Sat(\KKK^\sharp I^\sharp)$
by \eqref{compKatP}.

Further, since $I/J=(-1+P^\sharp)$ and $J/\KKK=(-1+3P^\sharp)$
have no global sections,
we have $\h^0(I)=\h^0(J)=\h^0(\KKK)$.
Take a global section $\beta$ of $\KKK$
such that $\beta \equiv y_4^2 \mod (y_1)$ at $P^\sharp$ by Lemma \ref{lemma-extend-sections-D}.
Then its image $\bar{\beta}$ in $\gr^3(\OOO,\KKK)$ is non-zero
because $\bar{\beta} \equiv z/c \mod y_1\KKK^\sharp$ by
$\beta - z/c \in (y_1)\cap \KKK^\sharp= y_1\KKK^\sharp$.
The image of $\bar{\beta}$ under the composition map
$\KKK \rightarrow \gr^3(\OOO,\KKK) \rightarrow (3P^\sharp)
\subset \gr^1_C \OOO$ is not zero
because $\h^0(\gr^3(\OOO,\KKK)) \rightarrow \h^0((3P^\sharp))
\to \h^0(\gr^1_C \OOO)$
is an isomorphism.

Then the equation $\beta=0$ defines a normal surface which is smooth outside of $P$.
Recall that $\gr^1_C\OOO=(3P^\sharp)\toplus (-1+2P^\sharp)$.
Since $y_2$ is an $\ell$-basis of $(3P^\sharp)\subset \gr^1_C\OOO$, we have $y_1^3y_2\in \beta$,
i.e. $\mu\neq 0$.
Since $y_4$ is an $\ell$-basis of the summand $(-1+2P^\sharp)$ and
this summand
has no global sections, $y_1^2y_4\notin\beta$.
Note that
$\gr^1(\OOO,\KKK)=(-1+2P^\sharp)$
and $\gr^2(\OOO,\KKK)=(-1+P^\sharp)$ by \eqref{Lsplitting-for-gr2OJ}.
Using the standard exact sequence similar to \eqref{equation-exact-sequence-F-omega}
we obtain $\chi(\omega/ F^3(\OOO,\KKK))=0$.
Hence by Lemma \ref{Lemma-fiber}(ii) the sequence \eqref{Lsplitting-for-gr3OK}
is not $\ell$-split.
Therefore, we have an $\ell$-splitting
\begin{equation*}
\gr^3(\OOO,\KKK)= (0)\toplus(-1+3P^\sharp).
\end{equation*}
By $\bar{\beta} \equiv z/c \in \gr^3(\OOO,\KKK)$, as above, we have
$\lambda=1/c\neq 0$.
Now apply Computation \ref{computation-lP=3a-new1}.
Since $c\lambda=a=1$, $\mu\neq 0$ and both $\delta$ and $\epsilon$ are general,
we see that only graphs \type{c_{s,r})}
are possible. On the other hand,
the whole configuration must be contractible, a contradiction.
\end{proof}

\section{Case $\ell(P)=4$.}\label{section-lP=4}
In this section we assume that $\ell(P)=4$ and $\h^0(\gr^1_C\OOO)\neq 0$.

\begin{de}
\label{(7.10)-b}
By Lemma \ref{lemma-new-ell-p}, $P$ is the only singular point of $X$.
Write
\[
\alpha=y_1^4y_4+y_2^2+y_3^2+\delta y_4^{2k+1}+\cdots,
\]
where $\delta\neq 0$ and $2k+1$ is the lowest power of $y_4$ appearing in $\alpha$.
We have $i_P(1)=2$ and $\deg\gr^1_C\OOO=-1$.
Hence,
\begin{equation}\label{equation-(7.4.1.1)-2b}
\gr^1_C\OOO=
\left\{
\begin{array}{lll@{\hbox{\hspace{70pt}}}l}
(P^\sharp)&\toplus&(-1+3P^\sharp),&\flabel{equation-(7.4.1.1)-2b}{a}
\\
(3P^\sharp)&\toplus&(-1+P^\sharp).&\flabel{equation-(7.4.1.1)-2b}{b}
\end{array}
\right.
\end{equation}
In the case \fref{equation-(7.4.1.1)-2b}{a} the germ $(X,C)$ is flipping
by \cite[(7.4.4)]{Kollar-Mori-1992}.
So we assume that \fref{equation-(7.4.1.1)-2b}{b} holds.
Then $y_2$ and $y_3$ form $\ell$-free $\ell$-bases of
$(3P^\sharp)$ and $(-1+2P^\sharp)$, respectively.
\end{de}

\begin{de}
Let $J$ be the ideal such that $I\supset J\supset I^{(2)}$ and
$J/I^{(2)}=(3P^\sharp)$.
Since $\gr^{2,1}(\OOO,J)= (-1+2P^\sharp)$ and $\gr^{2,0}(\OOO,J)= (3P^\sharp)$,
by the exact sequence \eqref{equation-exact-sequence} we have two possibilities:
\begin{equation}\label{equation-lP=4-grOJ}
\gr^2(\OOO,J)=
\left\{\begin{array}{l@{\hbox{\hspace{70pt}}}l}
(3P^\sharp)\toplus(-1+2P^\sharp),&\flabel{equation-lP=4-grOJ}{a}
\\
(2P^\sharp)\toplus(-1+3P^\sharp).&\flabel{equation-lP=4-grOJ}{b}
\end{array} \right.
\end{equation}

\begin{slemma}
The subcase {\fref{equation-lP=4-grOJ}{a}} does not occur.
\end{slemma}

\begin{proof}
Assume that we are in the situation \fref{equation-lP=4-grOJ}{a}.
Then the sequence \eqref{equation-exact-sequence}
splits. Hence, $(y_3,y_2,y_4)$ is a $(1,3,2)$-monomializing $\ell$-basis
\cite[8.11]{Mori-1988}. Take the ideal $\KKK$ so that
$J\supset \KKK\supset IJ$ and $\KKK/IJ=(3P^\sharp)$. Then we have
\begin{align*}
&\gr^1(\OOO,\KKK)&&=&&(-1+P^\sharp),\hspace{20pt}
&&\gr^2(\OOO,\KKK)&&=&&(-1+2P^\sharp),
\\
& \gr^{3,0}(\OOO,\KKK)&&=&&(3P^\sharp),
&&\gr^{3,1}(\OOO,\KKK)&&=&&(-2+3P^\sharp),
\\
&\gr^{4,0}(\OOO,\KKK)&&=&&(0),
&&\gr^{4,1}(\OOO,\KKK)&&=&&(-1).
\end{align*}
We get a contradiction by Lemma \ref{Lemma-fiber}(i).
\end{proof}

\begin{sde}{\bf Subcase \fref{equation-lP=4-grOJ}{b}.}
\label{generality-H-conic-bundle}
We can apply Corollary \ref{y12y4-in-beta} with $b=1$
and get that the coefficients of $y_4^2$ and $y_1^2y_4$ 
in $\beta$ are independent. 
Since the image of $\beta$ in $\gr^1_C\OOO$
is not zero, we have $y_1^3y_2\in \beta$.
Now we apply Computation \ref{computation-lP=4} and obtain \xref{main-theorem-conic-bundle-1.1.4}.
Since the configuration is
not birationally contractible, $f$ is a $\QQ$-conic bundle.
The existence 
can be shown similarly to Proposition \ref{proposition-existence}.
\end{sde}
\end{de}

\section{Case $\ell(P)\ge 5$.}\label{section-lP=5}
\begin{de}
\label{ellP= 5-(7.10)-b}
In this section we consider the case  $\ell(P)\ge 5$ and $\h^0(\gr^1_C\OOO)\neq 0$.
By Lemma \ref{lemma-new-ell-p}, $P$ is the only singular point of $X$
and $\ell(P)=5$.

\begin{sde}
We use the notation of \ref{(7.5)} at $P$.
In particular,
\[
\alpha=y_1^5y_2+y_2^2+y_3^2+\delta y_4^{2k+1}+cy_1^2y_4^2+\cdots.
\]
Since $\deg\gr^1_C\OOO=1-i_P(1)=-1$, we have the following possibilities:
\begin{equation}
\label{equation-possibilities-lP=5-gr1COOO}
\gr^1_C\OOO=
\left\{
\begin{array}{lll@{\hbox{\hspace{70pt}}}l}
(P^\sharp)&\toplus&(-1+2P^\sharp),
&\flabel{equation-possibilities-lP=5-gr1COOO}{a}
\\
(2P^\sharp)&\toplus&(-1+P^\sharp).
&\flabel{equation-possibilities-lP=5-gr1COOO}{b}
\end{array} \right.
\end{equation}
\end{sde}
\end{de}
\begin{de}{\bf Case \fref{equation-possibilities-lP=5-gr1COOO}{b}.}
\label{case-lP=5-b}
Then $\gr^1_C\OOO=(2P^\sharp)\toplus(-1+P^\sharp)$ and
$y_4$ (resp. $y_3$) is an $\ell$-basis of $(2P^\sharp)$ (resp. $(-1+P^\sharp)$) at $P$.
Further, $(-1+P^\sharp)^{\totimes 2}= (-2+2P^\sharp)$ with an $\ell$-basis $y_3^2$.
Let $J$ be the ideal such that $I\supset J\supset I^{(2)}$ and
$J/I^{(2)}=(2P^\sharp)$. Then $J^\sharp=(y_2,y_3^2,y_4)$ and
$\alpha\equiv y_1^5y_2+y_3^2\mod I^\sharp J^\sharp$.
Thus $y_3^2$ is divisible by $y_1^5$ and so
$\gr^{2,1}(\OOO,J)= (-1+3P^\sharp)=(-2+2P^\sharp+5P^\sharp)$.
We also have $\gr^{2,0}(\OOO,J)= (2P^\sharp)$.
Thus the exact sequence \eqref{equation-exact-sequence}
is $\ell$-split.
Now, as in \ref{generality-H-conic-bundle}, we apply Computation \ref{computation-lP=4}
with $(a,b)=(0,1)$
and obtain \ref{main-theorem-conic-bundle-1.1.4}.
The existence can be shown similarly to Proposition \ref{proposition-existence}.
\end{de}

\begin{de}{\bf Case \fref{equation-possibilities-lP=5-gr1COOO}{a}.}\label{case-lP=5-a}
Then $\gr^1_C\OOO=(P^\sharp)\toplus(-1+2P^\sharp)$
and
$y_3$ (resp. $y_4$) is an $\ell$-basis of $(P^\sharp)$ (resp. $(-1+2P^\sharp)$) at $P$.
Let $J$ be the ideal such that $I\supset J\supset I^{(2)}$ and
$J/I^{(2)}=(P^\sharp)$. Then $J^\sharp=(y_2,\, y_3,\, y_4^2)$.

\begin{sde}{\bf Subcase $y_1^2y_4^2\notin \alpha$.}
Then $y_1^5y_2\in I^\sharp J^\sharp$.
We can write
\[
\alpha\equiv y_1^5y_2+y_1^6y_4^2\, \gamma(y_1^4) \mod I^\sharp J^\sharp.
\]
Hence, $y_2+y_1y_4^2\gamma(y_1^4) \in \Sat( I^\sharp J^\sharp)$.
In the exact sequence \eqref{equation-exact-sequence} we have
$\gr^{2,1}(\OOO,J)=(-1+2P^\sharp)^{\totimes 2}=(-1)$
with an $\ell$-basis $y_4^2$
and
$\gr^{2,0}(\OOO,J)=(P^\sharp)$ with an $\ell$-basis $y_3$.
Hence, $y_4^2$ and $y_3$ form an $\ell$-basis of $\gr^2(\OOO,J)$.
If $\gr^2(\OOO,J)$ contains an $\ell$-direct summand $(-1)$,
then we get a contradiction by Lemma \ref{Lemma-fiber}(ii).
Thus $\gr^2(\OOO,J)= (0)\toplus(-1+P^\sharp)$.
Then the exact sequence \eqref{equation-exact-sequence}
is not $\ell$-split. An $\ell$-basis of the first summand $(0)$
can be written as $y_4^2+ \lambda y_1y_3$ for some $\lambda\in \OOO_{C,P}$.
By the above, $\lambda(P)\neq 0$. Hence, $y_1y_3\in \beta$.
The terms $y_4^2$ and $y_2y_3$ appear in $\beta$ by Lemma \ref{lemma-extend-sections-D}.
Then we can apply Computation \ref{computation-lP=3a-new1} in which $(a,b)=(0,1)$,
$c=0$, and $\lambda\ne 0$.
Since $c=0$, the only case \ref{computation-lP=3a-new1} \type{c_{s,r})}
is possible.
But this contradicts \ref{sde1}.
\end{sde}

\begin{sde}{\bf Subcase $y_1^2y_4^2\in \alpha$.}
Then $y_1^5y_2+cy_1^2y_4^2\in I^\sharp J^\sharp$ with $c\neq 0$.
Hence,
\[
y_1^3y_2+cy_4^2\in \Sat(I^\sharp J^\sharp).
\]
In the exact sequence \eqref{equation-exact-sequence} we have
\[
\gr^{2,1}(\OOO,J)=(-1+3P^\sharp)
=(-1+2P^\sharp)^{\totimes 2}\totimes (3P^\sharp)
\supset (-1+2P^\sharp)^{\totimes 2}
\]
with an $\ell$-basis $y_2$.
So, the sequence \eqref{equation-exact-sequence}
is $\ell$-split with $\gr^{2,0}(\OOO,J)=(P^\sharp)$.
We can apply Corollary \ref{y12y4-in-beta} with 
$b=2$. 
Hence we can apply Computation \ref{computation-lP=3a-new1}
in which $(a,b)=(0,1)$, $c\neq 0$, and $\lambda$ is general.
Then we obtain \ref{main-theorem-1.1.3}.
The following example shows that this case does occur
(cf. Proposition \ref{proposition-existence}).
\end{sde}
\end{de}

\begin{sde}{\bf Example.}\label{example-1-1-3a}
As in \ref{example-1-1-3}, let $Z \subset {\mathbb C}^5_{z_1,\ldots,z_5}$ be
defined by
\begin{eqnarray*}
0&=& z_2^2+z_3+z_4z_5^k+z_1^2z_5,\qquad k\ge 1,\\
0&=& z_1^2z_2^2 + z_4^2-z_3z_5+z_1^3z_2.
\end{eqnarray*}
The origin
of the $z_3$-chart
is a type \type{(IIA)} point $P$ with $\ell(P)=5$:
\[
\{y_2^2+y_3^2+y_1^2(y_1^2y_2^2+y_4^2+y_1^3y_2) +y_4(y_1^2y_2^2+y_4^2+y_1^3y_2)^k=0\}/\muu_{4}(1,1,3,2),
\]
where $(C,P)$ is the $y_1$-axis.
Moreover, $X$ is smooth outside of $P$.
Thus $X\to Z$ is a divisorial contraction of type \ref{main-theorem-1.1.3}.
\end{sde}

\section{Appendix: resolution of certain surface singularities}

\begin{assumption}\label{notation-blowup}
Let
$W:= \CC^4_{y_1,\dots,y_4}/\muu_4(1,1,3,2)$,
let $C:=\{\text{$y_1$-axis}\}/\muu_4$, and let $\sigma$ be the weight $\frac14(1,1,3,2)$.
Consider a normal
surface singularity $H\ni 0$ given in $W$ by
two $\sigma$-semi-invariant equations $\alpha=\beta=0$.
We assume that the following conditions are satisfied.
\begin{itemize}
\item
$H$ contains $C$ and is smooth outside of $C$,
\item
$\wt \alpha\equiv 2\mod 4$,\quad
$\wt \beta\equiv 0\mod 4$,
\item
$\alpha\equiv y_1^{l}y_j\mod (y_2, y_3, y_4)^2$ for some $j\in \{2,3,4\}$ and $l>0$,
\item
$\alpha_{\sigma=2/4}=y_2^2$, \quad $y_3^2\in \alpha$,
\item
$y_4^2$ appears in $\beta$ with coefficient $1$,
\item
$y_2y_3$ appears in $\beta$ with coefficient $\nu$ which can be taken general,
\item
$H$ has only rational singularities
and, for any resolution, the total transform of $C$ has only simple normal crossings.
\end{itemize}

\begin{sde}\label{computations-notation}
We can write the equations in the following form
\begin{eqnarray*}
\alpha&=&y_1^ly_j+y_2^2+y_3^2+\delta y_4^{2k+1}+c y_1^2y_4^2+\epsilon y_1y_3y_4
+y_2\alpha'+\alpha'',
\\
\beta&= &y_4^2+\nu y_2y_3+\lambda y_1y_3+\mu y_1^3y_2+\eta y_1^2y_4+y_2\beta'+\beta'',
\end{eqnarray*}
where
$\delta, c, \epsilon, \mu, \eta, \nu, \lambda$ are constants,
$\alpha',\,\beta',\, \beta''\in (y_2,\, y_3,\, y_4)$,\
$\alpha''\in (y_2,\, y_3,\, y_4)^2$,\
$\sigmaord (\beta')= 3/4$,
$\sigmaord (\alpha')= 5/4$,
$\sigmaord (\beta'')> 1$,
$\sigmaord (\alpha'')> 3/2$, and $2k+1$ is the smallest exponent of $y_4$
appearing in $\alpha$.
We usually assume that all the summands in $\alpha$ (resp. $\beta$)
have no common terms. Then
$\beta'\in (y_1y_4,\, y_1y_2,\, y_2^2, y_2y_4)$.
\end{sde}
\end{assumption}

\begin{sconstruction}\label{notation-computations--sing}
Consider the weighted $\sigma$-blowup $\Phi: \tilde W\to W$. Let
$\tilde H$ be the proper transform of $H$ on $\tilde W$ and
$\Pi\subset \tilde W$ be the $\Phi$-exceptional divisor. Then $\Pi\simeq \PP(1,1,3,2)$ and
$\OOO_{\Pi}(\Pi)\simeq \OOO_{\PP}(-4)$.
Put
\begin{equation}\label{equation-Lambda-computation-2}
\begin{array}{l}
O:=(1:0:0:0),\ Q:=(0:0:1:0)\in \Pi,
\\[5pt]
\Lambda:=\{y_2=\alpha_{\sigma=6/4}=0\}\subset \Pi.
\end{array}
\end{equation}
Let $\tilde X\subset \tilde W$
(resp. $\tilde C\subset \tilde W$) be the proper transform of $X:=\{\alpha=0\}$
(resp. $C$).
Clearly, $\tilde C\cap \Pi=\{ O\}$.
Denote (scheme-theoretically)
\[
\Xi:=\tilde H\cap \Pi = \{y_2^2=\beta_{\sigma=1}=0\} \subset \Pi.
\]
\end{sconstruction}

\begin{sclaim}\label{claim-1-construction-blowup}
Any irreducible component $\Xi_i$ of $\Xi$ is a smooth rational curve
passing through $Q$.
Moreover, $\Xi=2\Xi_1$ \textup(resp. $\Xi=2\Xi_1+2\Xi_2$,\ $\Xi=4\Xi_1$\textup)
if and only if $\lambda\neq 0$ \textup(resp. $\lambda=0$ and $\eta\neq 0$,
$(\lambda, \eta)= (0,0)$\textup).
\end{sclaim}

\begin{sclaim}\label{claim-2-construction-blowup}
The point $Q\in \tilde H$ is Du Val of type \type{A_2}.
In particular, $\tilde H$ is normal.
\end{sclaim}
\begin{proof}
In the affine chart
$\CC^4/\muu_3(1,1,2,2)\simeq \{y_3\neq 0\}\subset \tilde W$
the surface $\tilde H$ is quasismooth at the origin.
More precisely, it is locally isomorphic to $\CC^2_{y_1,y_4}/\muu_3(1,2)$,
that is, $Q\in \tilde H$ is of type \type{A_2}.
Since $\tilde H\cap \Pi=\Xi$, $\tilde H$ is normal.
\end{proof}

\begin{sclaim}\label{claim-3-construction-blowup}
If $\Xi=2\Xi_1+2\Xi_2$, then the pair $(\tilde H, \Xi_1+\Xi_2)$
is not lc at $Q$ and lc outside of $Q$.
\end{sclaim}
\begin{proof}
Since $\Xi_1$ and $\Xi_2$ are tangent to each other, the pair $(\tilde H, \Xi_1+\Xi_2)$
is not lc at $Q$.
Since
\[
(K_{\tilde H}+\Xi_1+ \Xi_2)\cdot \Xi_i=-\frac 14\Xi\cdot \Xi_i=\frac 13,
\]
as in the proof of Lemma \ref{label-surfaces-C} we have
\[
\deg \Diff_{\Xi_1}(\Xi_2)=-\deg K_{\Xi_1}+\frac 13=\frac 73,
\qquad
\deg \Diff_{\Xi_2}(\Xi_1)=\frac 73,
\]
and coefficients of $\Diff_{\Xi_1}(\Xi_2)$ and $\Diff_{\Xi_2}(\Xi_1)$ at
$Q$ are $\ge 4/3$. Hence,
all other coefficients are $\le 1$ and so
the pair $(\tilde H, \Xi_1+ \Xi_2)$ is lc outside of
$Q$.
\end{proof}

\begin{sclaim}\label{claim-construction-blowup-Du-Val}
The singularities of $\tilde H$ are Du Val.
\end{sclaim}
\begin{proof}
By Claim \ref{claim-2-construction-blowup}
\ $\tilde H$ has a Du Val singularity at $Q$.
Outside of $Q$ the surface $\tilde H$ has only 
complete intersection rational singularities.
Hence, they are Du Val. 
\end{proof}

\begin{sclaim}\label{claim-4-construction-blowup}
$\Sing(\tilde X)\cap \Pi$ consists of the curve $\Lambda$, the point $Q$, and 
the point $(0:0:0:1)$ \textup(only if $k>1$\textup).
\end{sclaim}

\begin{sclaim}\label{claim-5-equation-notation-blowup}
$K_{\tilde H}=\Phi^*K_H-\frac 34 \Xi$.
\end{sclaim}
\begin{proof}
Follows by the adjunction and because $K_{\tilde W}=\Phi^*K_W+\frac 34 \Pi$.
\end{proof}

\begin{sclaim}\label{claim-computation-Xi}
Let $\varphi: \hat H\to \tilde H$ be
the minimal resolution
and let $\hat \Xi_i\subset \hat H$ be
the proper transform of $\Xi_i$.
\begin{itemize}
\item
If $\Xi=2\Xi_1$, then $\hat \Xi_1^2= -4$.
\item
If $\Xi=2\Xi_1+2\Xi_2$, then $\hat \Xi_i^2= -3$.
\end{itemize}
\end{sclaim}

\begin{proof}
By Claim \ref{claim-construction-blowup-Du-Val} the surface
$\tilde H$ has only Du Val singularities.
Hence, $K_{\hat H}=\varphi^* K_{\tilde H}$.
Using \ref{claim-5-equation-notation-blowup} we can write
\[
K_{\hat H}\cdot \hat \Xi_i=K_{\tilde H}\cdot \Xi_i=-\frac34 \Xi\cdot \Xi_i=
-\frac34 \Pi\cdot \Xi_i= -\frac34 \OOO_{\PP}(-4)\cdot \Xi_i=3\OOO_{\PP}(1)\cdot \Xi_i.
\]
If $\Xi=2\Xi_1$, then $\Xi_1$ is given by $y_2=\beta_{\sigma=1}=0$
and so $K_{\hat H}\cdot \hat \Xi_1=2$, i.e. $\hat \Xi_1$ is a $(-4)$-curve.
If $\Xi=2\Xi_1+2\Xi_2$, then similarly $K_{\hat H}\cdot \hat \Xi_i=1$,
i.e. $\hat \Xi_i$ is a $(-3)$-curve.
\end{proof}

\begin{computation}\label{computation-lP=3a-new1}
In the notation of \xref{notation-blowup}, let $H\ni 0$ be a normal singularity
with
\begin{eqnarray*}
\alpha&=&ay_1^3y_3+by_1^5y_2+y_2^2+y_3^2+\delta y_4^3+c y_1^2y_4^2+\epsilon y_1y_3y_4
+y_2\alpha'+\alpha'',
\\
\beta&= &y_4^2+\nu y_2y_3+\lambda y_1y_3+\mu y_1^3y_2+y_2\beta'+\beta'',
\end{eqnarray*}
where
$(a,b)= (1,0)$ or $(0,1)$, $(b,\mu)\neq (0,0)$,
$(c\lambda, c\mu)\neq (a,b)$, and $\lambda\neq 0$.
We assume that the hypothesis of \xref{computations-notation}
are satisfied.
Then the graph $\Delta(H,C)$ has
one of the following forms:

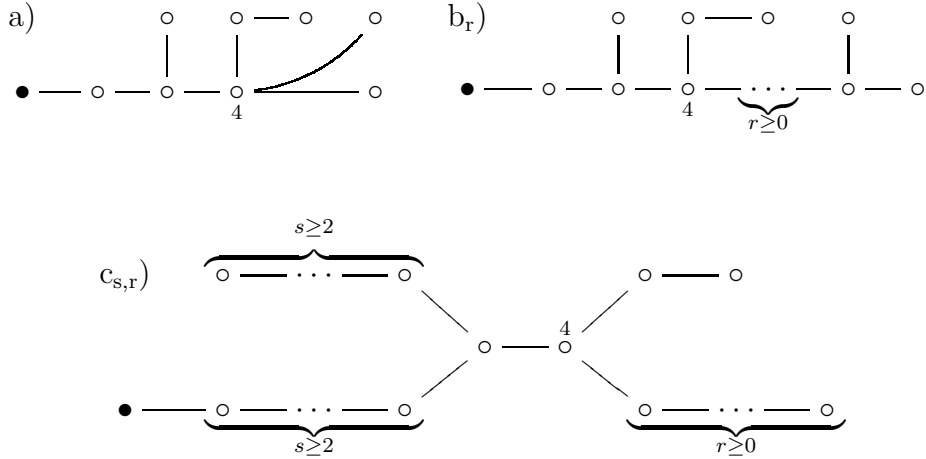
\begin{equation*}
\xy
\xymatrix @R=8pt@C=13pt{
\mathrm{a})&&\circ\ar@{-}[d]&\circ\ar@{-}[r]\ar@{-}[d]&\circ&\circ\ar@/^7pt/@{-}[dll]
\\
\bullet\ar@{-}[r]&\circ\ar@{-}[r]&\circ\ar@{-}[r]
&\underset4\circ\ar@{-}[rr]&&\circ
}
\endxy
\hspace{17pt}
\xy
\xymatrix @R=7pt@C=13pt{
\mathrm{b_r)}&&\circ\ar@{-}[d]&\circ\ar@{-}[r]\ar@{-}[d]&\circ&\circ\ar@{-}[d]
\\
\bullet\ar@{-}[r]&\circ\ar@{-}[r]&\circ\ar@{-}[r]
&\underset4\circ\ar@{-}[r]&\cdots\POS*\frm{_\}},-U*---!D\txt{$\scriptstyle{r\ge 0}$}
&\circ\ar@{-}[r]\ar@{-}[l]&\circ
}
\endxy
\end{equation*}
\[
\vcenter{
\xy
\xymatrix"M"@R=7pt@C=17pt{
\mathrm{c_{s,r})}&\circ\ar@{-}[r]&\cdots&\ar@{-}[l]\circ&&&\circ\ar@{-}[r]&\circ
\\
&&&&\circ\ar@{-}[lu]\ar@{-}[ld]&\overset 4\circ\ar@{-}[ru]\ar@{-}[l]\ar@{-}[rd]
\\
\bullet\ar@{-}[r]&\circ\ar@{-}[r]&\cdots&\ar@{-}[l]\circ&&&\circ\ar@{-}[r]&\cdots&\circ\ar@{-}[l]
}
\POS"M3,7"."M3,9"!C*\frm{_\}},-U*---!D\txt{$\scriptstyle{r\ge 0}$}
\POS"M3,2"."M3,4"!C*\frm{_\}},-U*---!D\txt{$\scriptstyle{s\ge 2}$}
\POS"M1,2"."M1,4"!C*\frm{^\}},-U*++++++!D\txt{$\scriptstyle{s\ge 2}$}
\endxy
}
\vspace{16pt}
\]

\noindent
Moreover, \textup{a)} occurs if and only if
$c\lambda\neq a$ and
$\lambda(\lambda \delta-\epsilon)^2\neq 4 (c\lambda-a)$, and
\type{b_r)} occurs if and only if
$c \lambda\neq a$ and
$\lambda(\lambda \delta-\epsilon)^2= 4 (c\lambda-a)$.
The case \type{c_{s,r})} occurs only with $(s,r)=(2,1)$ or $(3,0)$.
\end{computation}

\begin{proof}
We use the notation of \xref{notation-blowup}.
By \ref{claim-1-construction-blowup} we have
$\Xi=2\Xi_1$, where $\Xi_1:=\{y_2=y_4^2+\lambda y_1y_3=0\}$.
By Claim \ref{claim-construction-blowup-Du-Val} $\Delta(H,C)$ contains
at most one vertex of weight $\ge 3$. This vertex corresponds to $\Xi_1$
and its weight equals $3$ by Claim \ref{claim-computation-Xi}.

\begin{sclaim}\label{sclaim-11-2-O}
The point $O\in \tilde H$ is of type \type{A_n}, where $n\ge 3$ and
$n=3$ if and only if $c \lambda\neq a$.
Furthermore, the proper transform $\tilde C$ of $C$ meets $\Xi_1$ at $O$,
the pair $(\tilde H,\tilde C)$ is plt and the pair $(\tilde H,\Xi_1)$ is not plt
at $O$.
\end{sclaim}

\begin{proof}
Recall that by Claim \ref{claim-construction-blowup-Du-Val} the
singularities of $\tilde H$ are Du Val.
In the affine chart $\CC^4\simeq U_1=\{y_1\neq 0\}\subset \tilde W$
the equations of $\tilde H$ have the following form
\begin{eqnarray*}
&&ay_1y_3+by_1y_2+y_2^2+c y_1y_4^2+y_1\alpha_\bullet=0,
\\
&&y_4^2+\lambda y_3+\mu y_2+\beta_\bullet=0,
\end{eqnarray*}
where $\alpha_\bullet$, $\beta_\bullet\in (y_1,y_2,y_3,y_4)^2$
and $y_4^2\notin \alpha_\bullet$, $\beta_\bullet$.
Now consider the usual blowup $\bar H\to \tilde H$ of the origin
$O\in \tilde H\cap U_1$. Let $\Pi_1$ be the exceptional divisor of the
ambient space. Then $\Pi_1\simeq \PP^3$.
The exceptional curve $\Theta:=\bar H\cap \Pi_1$ of our surface is given in $\Pi_1$ by
\[
\lambda y_3+\mu y_2=y_2\bigl((b\lambda-a\mu)y_1+y_2\bigr)=0.
\]
Since $(b,\mu)\neq (0,0)$, we have $b\lambda\neq a\mu$ and so
$\Theta=\Theta_1+\Theta_2$,
where $\Theta_1:=\{y_2=y_3=0\}$ and $\Theta_2=\{(b\lambda-a\mu)y_1+y_2=\lambda y_3+\mu y_2=0\}$
with $\Theta_1\cap \Theta_2=\{Q_1\}$, $Q_1:=(0:0:0:1)$.
Hence $\bar H$ is smooth outside of $Q_1$.
Moreover,
the proper transform of $\Xi_1$ passes through $Q_1$, and
the proper transform of $C$ meets $\Theta_1$ at $O_1:=(1:0:0:0)$.
{}From the classification of Du Val singularities we immediately see
that $O\in \tilde H$ is of type \type{A_n} with $n\ge 2$.
In the chart $y_4\neq 0$
the surface $\bar H$ is given by the following equations:
\[
\begin{array}{rcl}
ay_1y_3+by_1y_2+y_2^2+c y_1y_4+(\text{terms of degree $\ge 3$})&=&0,
\\[5pt]
y_4+\lambda y_3+\mu y_2+(\text{terms of degree $\ge 2$})&=&0.
\end{array}
\]
Eliminating $y_4$ we get
\[
(a-c \lambda)y_1y_3+y_2^2+ (b- c\mu)y_1y_2 +(\text{terms of degree $\ge 3$}).
\]
{}From this, one can see that $\bar H\ni Q_1$ is singular of type \type{A}. Moreover,
it is of type \type{A_1}
if and only if $a\neq c \lambda$.
This proves our claim.
\end{proof}

\begin{sclaim}\label{sclaim-11-2-Lambda}
The set $\Sing(\tilde H) \setminus \{ Q\}$
coincides with $\Lambda\cap \Xi_1$ and we have the following possibilities:
\[
\# \Sing(\tilde H)=
\begin{cases}
2&\text{if $c\lambda=a$ and $\lambda \delta=\epsilon$,}
\\
3 & \text{if $c\lambda=a$ and $\lambda \delta\neq\epsilon$ or}
\\
&\text{$c\lambda\neq a$ and $\lambda(\lambda \delta-\epsilon)^2=4 (c\lambda-a)$,}
\\
4& \text{if $c\lambda\neq a$ and
$\lambda(\lambda \delta-\epsilon)^2\neq 4 (c\lambda-a)$.}
\end{cases}
\]
\end{sclaim}

\begin{proof}
Clearly, $\Sing(\tilde H) \setminus \{ Q\} \supset \Lambda\cap \Xi_1$.
By writing down explicit equations one can see that these two sets coincide
and are given in $\Pi$ by
\[
y_2= y_4^2+\lambda y_1y_3=ay_1^3y_3+y_3^2+\delta y_4^3+c y_1^2y_4^2+\epsilon y_1y_3y_4=0.
\]
Since $\lambda\neq 0$, in the affine chart $y_1\neq 0$ this can be rewritten as follows:
\begin{equation}
\label{equation-lP=3-intersection=Xi=Lambda}
y_2=y_4^2+\lambda y_3=0, \ \
y_4^2\bigl(y_4^2+\lambda(\lambda \delta-\epsilon) y_4+\lambda(c\lambda-a)\bigr)=0.
\end{equation}
Then we get the possibilities as claimed.
\end{proof}

\begin{sclaim}
If $\# \Sing(\tilde H)=4$, then $\Delta(H,C)$ is of type \textup{a)}.
If $\# \Sing(\tilde H)\le 3$ and either
$(\tilde H,\Xi_1)$ is not plt outside of $O$
or
lc at $O$, then $\Delta(H,C)$ is of type \type{b_r)}.
\end{sclaim}

\begin{proof}
As in the proof of Lemma \ref{label-surfaces-C} we have
\begin{equation}
\label{equation-Diff-Xi-computation-2}
\deg \Diff_{\Xi_1}(0)=-\deg K_{\Xi_1}+(K_{\tilde H}+\Xi_1)\cdot \Xi_1=\textstyle\frac 83,
\end{equation}
the coefficient of $\Diff_{\Xi_1}(0)$ at $Q$ equals $2/3$, and
the coefficient of $\Diff_{\Xi_1}(0)$ at $O$ is $\ge1$.
If $\# \Sing(\tilde H)=4$, then the only possibility is
$\Diff_{\Xi_1}(0)=O+ \frac23 Q+ \frac12P_1+ \frac12 P_2$.
Hence, $(\tilde H,\Xi_1)$ is plt and $\Delta(H,C)$ is of type \textup{a)}.
Similarly, in the second case the only possibility is $\Diff_{\Xi_1}(0)=O+ \frac23 Q+ P_1$
and so $\Delta(H,C)$ is of type \type{b_r)} (cf. \cite[ch. 3]{Utah}).
\end{proof}

{}From now on we assume that $\# \Sing(\tilde H)\le 3$,
$(\tilde H,\Xi_1)$ is plt outside of $O$,
and $(\tilde H,\Xi_1)$ is not lc at $O$.
Thus $\tilde H$ has at most one singularity $P_1$ outside of $\{O,\, Q\}$
and this singularity is of type \type{A_r} for some $r\ge 0$.
By the above $\tilde H\ni O$ is a point of type \type {A_n}
with $n>3$ (because $(\tilde H,\Xi_1)$ is not lc at $O$).
Thus the graph $\Delta(H,C)$ looks as follows:
\[
\xy
\xymatrix@R=7pt@C=10pt{
&\Delta_Q\ar@{-}[d] &
\\
\Delta_{O} \ar@{-}[r] &\hat \Xi_1 \ar@{-}[r]& \Delta_{P_1}
}
\endxy
\]
where $\Delta_{P_1}$, $\Delta_Q$, $\Delta_O$ are chains corresponding to resolution of
points $P_1$, $Q$, $O$, respectively.
Moreover, $\Delta_{P_1}$ and $\Delta_Q$
adjacent to $\hat \Xi_1$ by end vertices.
Note that $2\Xi_1=\Xi$ is a Cartier divisor on $\tilde H\setminus\{Q\}$
and $\Xi_1$ is smooth.
By the classification of Du Val singularities this implies that
$n$ is odd and
the proper transform of $\Xi_1$ on the minimal resolution of $O\in \tilde H$ meets
the middle curve in the chain of $\Delta_O$.
This gives \type{c_{s,r})}. Moreover, in this case
\[\textstyle
\Diff_{\Xi_1}(0)=\frac 23 Q+ \frac {s+1}2 O +\frac {r}{r+1} P_1, \quad \frac 83=\frac 23 + \frac {s+1}2 +\frac {r}{r+1}
\]
and so $r+s=3$. This completes the proof of \ref{computation-lP=3a-new1}.
\end{proof}

\begin{sremark}
Let $\alpha$ and $\beta$ be as in \ref{computation-lP=3a-new1}
with $(a,b)=(1,0)$ or $(0,1)$. Assume that the coefficients $\delta$, $c$, \dots
are sufficiently general. Then $\Delta(H,C)$ is of type \type{a)}.
Indeed, by Bertini's theorem $\tilde H$ is smooth outside of
$O$, $Q$ and $\Xi_1\cap \tilde H$. The singularity of $\tilde H$ at $O$ is of type
\type{A_3} by \ref{sclaim-11-2-O}, at $Q$ is of type
\type{A_2} by \ref{claim-2-construction-blowup}.
The intersection $\Xi_1\cap \tilde H$ consists of two points of type
\type{A_1} by \ref{sclaim-11-2-Lambda}.
\end{sremark}

\begin{computation}\label{computation-lP=4}
In the notation of \xref{notation-blowup}, let $H$ be a normal singularity with
\begin{eqnarray*}
\alpha&=&ay_1^4y_4+by_1^5y_2+y_2^2+y_3^2+\delta y_4^3+c y_1^2y_4^2+\epsilon y_1y_3y_4+
y_2\alpha'+\alpha'',
\\
\beta&= &y_4^2+\nu y_2y_3+\eta y_1^2y_4+\mu y_1^3y_2+
y_2\beta'+\beta'',
\end{eqnarray*}
where $\eta\neq 0$,
$(a,b)=(1,0)$ or $(0,1)$, and
$(b,\mu)\neq (0,0)$.
Furthermore, assume that $\eta$ is general with respect to 
the coefficients of $\alpha$.
We assume that the hypothesis of \xref{computations-notation}
are satisfied.
Then the graph $\Delta(H,C)$ has
the following form:

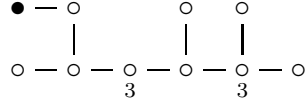
\begin{equation*}
\xy
\xymatrix@R=7pt@C=8pt{
&\bullet\ar@{-}[r]&\circ\ar@{-}[d]&&\circ\ar@{-}[d]&\circ\ar@{-}[d]
\\
&\circ\ar@{-}[r]&\circ\ar@{-}[r]
&\underset3\circ\ar@{-}[r]&\circ\ar@{-}[r]&\underset3\circ \ar@{-}[r]&\circ
}
\endxy
\end{equation*}
\end{computation}

\begin{proof}
We use the notation \xref{notation-blowup}.
In our case
$\Xi=2\Xi_1+2\Xi_2$, where
\[
\Xi_1=\{y_2=y_4=0\},\qquad
\Xi_2=\{y_2=y_4+\eta y_1^2=0\},
\]
$O\in \Xi_1$, $O\notin \Xi_2$,
and $\Xi_1\cap \Xi_2=\{ Q\}$.
Recall that by Claim \ref{claim-construction-blowup-Du-Val}
$\tilde H$ has only Du Val singularities.

\begin{sclaim}
$\tilde H$ has an \type{A_3}-point at $O\in \Xi_1$.
Furthermore, the proper transform $\tilde C$ of $C$ meets $\Xi_1\cup \Xi_2$ transversely at $O$,
the pair $(\tilde H,\tilde C)$ is plt,
the pair $(\tilde H,\Xi_1)$ is not plt at $O$,
and $\tilde H$ is smooth on $\Xi_1\setminus \{O,\, Q\}$.
\end{sclaim}

\begin{proof}
In the affine chart $\CC^4\simeq U_1=\{y_1\neq 0\}\subset \tilde W$
the surface $\tilde H$ is given by
\begin{eqnarray*}
ay_1y_4+by_1y_2+y_2^2+y_1\alpha_\bullet&=&\eta y_4+\mu y_2+\beta_\bullet=0,
\end{eqnarray*}
where $\alpha_\bullet, \, \beta_\bullet\in (y_1, y_2, y_3, y_4)^2$.
Consider the usual blowup $\bar H\to \tilde H$ of the origin
$O\in \tilde H\cap U_1$. Let $\Pi_1\simeq \PP^3$ be the exceptional divisor of the
ambient space.
The exceptional curve $\Theta:=\bar H\cap \Pi_1$ is given in $\Pi_1$ by
\[
(b\eta -a\mu )y_1y_2+y_2^2=\eta y_4+\mu y_2=0.
\]
Since $(b,\mu)\neq (0,0)$, we see that $\Theta=\Theta_1+\Theta_2$
is reduced and has two irreducible components meeting at one point
$Q_1:=(0:0:1:0)$.
{}From the classification of Du Val singularities we immediately see
that $O\in \tilde H$ is of type \type{A}.
Moreover,
the proper transform of $\Xi_1$ passes through $Q_1$ and
the proper transform of $C$ meets only one of the components of $\Theta$.
Hence the pair $(\tilde H,\Xi_1)$ is not plt at $O$
and the pair $(\tilde H,\tilde C)$ is plt.
Since the pair $(\tilde H,\Xi_1)$ is lc at $O$, $(\tilde H, O)$ is an \type{A_3}-point.
Finally by Lemma \ref{label-surfaces-C} $\tilde H$ has no singular
points on $\Xi_1\setminus \{O,\, Q\}$.
\end{proof}
The above claim gives a complete description of the left hand side of $\Delta(H,C)$
(adjacent to the left $\underset 3\circ$).
Since $\Xi_1\cap \Xi_2=\{Q\}$, the the middle part of the graph
has the desired form.

\begin{sclaim}
$\Sing(\tilde H)\cap \Xi_1 =\{O,\, Q\}$ and
$\Sing(\tilde H)\cap \Xi_2 = \{Q\}\cup (\Lambda\cap \Xi_2)$, where for
$\Lambda\cap \Xi_2=\{P_1,\, P_2\}$, $P_1\neq P_2$
and both points $(\tilde H,P_i)$ are of type \type{A_1}.
\end{sclaim}

\begin{proof}
Clearly, $\Sing(\tilde H) \supset \Lambda\cap \tilde H$ and
$\Sing(\tilde H) \ni Q$.
{}From \eqref{equation-Lambda-computation-2}
we obtain that $\Lambda\cap \Xi_1 =\{O,\, Q\}$ and $\Lambda\cap \Xi_2$
is given by
\begin{equation}\label{equation-Lambda-cap-H-2}
y_2=y_3^2-\epsilon \eta y_1^3 y_3
+(c \eta-a -\delta\eta^2 )\eta y_1^6 =
y_4+\eta y_1^2=0.
\end{equation}
Since $\eta$ is general, one can see that $\Lambda\cap \Xi_2$
consists of two points.
\end{proof}
Now we are ready to complete the proof of \ref{computation-lP=4}.
Since $K_{\tilde H} \cdot \Xi_i=1$
and $\Xi_1\cdot \Xi_2=2/3$ we have $\Xi_i^2=-4/3$ and
\[\textstyle
\deg \Diff_{\Xi_i}(0)= (K_{\tilde H}+\Xi_i)\cdot \Xi_i -\deg K_{\Xi_i} =\frac 53.
\]
Therefore, by the last claim we have
\[\textstyle
\Diff_{\Xi_2}(0)= \frac 23 Q+ \frac 12 P_1+\frac 12 P_2.
\]
Thus $(\tilde H,\Xi_2)$ is plt.
This completes the proof of \ref{computation-lP=4}.
\end{proof}

\par\medskip\noindent
{\bf Acknowledgments.}
The paper was written during the second author's visits to RIMS, Kyoto University.
The author is very grateful to the institute for
invitations, support, and hospitality.

\def\mathbb#1{\mathbf#1} \def\bblapr{April}

\end{document}